\documentclass[12pt]{article}
\usepackage{a4wide, amstext, amsmath, amssymb, graphicx, array, epsfig, psfrag}
\usepackage{amstext, amsmath, amssymb, graphicx, todonotes}
\usepackage{stmaryrd,color}
\usepackage[update,prepend]{epstopdf}

\newcommand{\bfn}{\boldsymbol n}
\newcommand{\bfI}{\boldsymbol I}
\newcommand{\bfP}{\boldsymbol P}

\newcommand{\bfx}{\boldsymbol x}
\newcommand{\bfa}{\boldsymbol a}

\newcommand{\bfp}{\boldsymbol p}

\newcommand{\bfv}{\boldsymbol v}
\newcommand{\bfu}{\boldsymbol u}
\newcommand{\bfw}{\boldsymbol w}

\newcommand{\bfB}{\boldsymbol B}

\newcommand{\bfg}{\boldsymbol g}
\newcommand{\bfe}{\boldsymbol e}

\newcommand{\bfkappa}{\boldsymbol \kappa}

\newcommand{\bfphi}{\boldsymbol \phi}

\newcommand{\bfpsi}{\boldsymbol{\psi}}
\newcommand{\bfzero}{\boldsymbol 0}

\newcommand{\bfV}{\boldsymbol V}

\newcommand{\mcK}{\mathcal{K}}
\newcommand{\mcE}{\mathcal{E}}

\newcommand{\tn}{|\mspace{-1mu}|\mspace{-1mu}|}

\newcommand{\Gammah}{{\Gamma_h}}
\newcommand{\nablas}{\nabla_\Gamma}
\newcommand{\nablash}{\nabla_{\Gamma_h}}

\newcommand{\divs}{\text{\rm{div}}_\Gamma}
\newcommand{\divsh}{\text{\rm{div}}_\Gammah}
\newcommand{\IR}{\mathbb{R}}
\newcommand{\bfPs}{\boldsymbol{P}_\Gamma}
\newcommand{\bfPsh}{\boldsymbol{P}_\Gammah}
\newcommand{\bfQs}{\boldsymbol{Q}_\Gamma}

\numberwithin{equation}{section}

\newtheorem{lem}{Lemma}[section]

\newtheorem{thm}{Theorem}[section]
\newtheorem{rem}{Remark}[section]
\newtheorem{cor}{Corollary}[section]

\newenvironment{proof}{\noindent \newline {\bf Proof.}}
{\hfill \mbox{\fbox{} } \newline}
\begin{document}
\title{\bf A Stabilized Finite Element Method for the 
Darcy Problem on Surfaces
\thanks{This research was supported in part by the Swedish Foundation for Strategic Research Grant No.\ AM13-0029, the Swedish Research Council Grants No.\ 2011-4992 and No.\ 2013-4708, and Swedish strategic research programme eSSENCE.}
}
\author{
Peter Hansbo\footnote{Department of Mechanical Engineering, J\"onk\"oping University, SE--551~11   J\"onk\"oping, Sweden, Peter.Hansbo@jth.hj.se } \mbox{ }
{Mats~G.~Larson} \footnote{Department of Mathematics and Mathematical Statistics, Ume{\aa} University, SE--901~87~~Ume{\aa}, Sweden, mats.larson@math.umu.se}
}
\maketitle

\begin{abstract} We consider a stabilized finite element method 
for the Darcy problem on a surface based on the Masud-Hughes formulation. A special feature of the method is that the 
tangential condition of the velocity field is weakly enforced 
through the bilinear form and that standard parametric continuous polynomial spaces on triangulations can be used. We prove 
optimal order a priori estimates that take the approximation 
of the geometry and the solution into account.
\end{abstract}

\section{Introduction}

In this note we develop a stabilized finite element method 
for Darcy flow on triangulations of a smooth surface. Starting 
from the Masud-Hughes formulation \cite{MaHu02} we obtain a 
very convenient method based on a coercive bilinear form that 
can handle different approximation spaces. More precisely, 
we consider parametric continuous piecewise polynomial 
elements, with possibly different orders in the mapping 
as well as the spaces for the velocity and pressure. A special 
feature of our approach is that we avoid using vector elements 
and discretize the tangent velocity vector componentwise in 
$\IR^3$ together with a weak enforcement of the tangent 
condition. Our approach is in contrast with the recent report 
\cite{FeFoFu14} where a method for Darcy flow based on 
Raviart-Thomas spaces was presented.  

We derive error estimates that takes the approximation of 
the geometry and the solution into account and separates the 
dependency of the different orders of approximations. The 
error in the velocity is defined using standard componentwise 
liftings based on the closest point mapping and we show an energy 
estimate, an $L^2$ estimate for the pressure, and an
$L^2$ estimate for the tangential part of the velocity which is 
slightly sharper with respect to the geometry approximation 
compared to the bound for the full velocity vector provided by 
the energy norm estimate. We also provide numerical results 
confirming our theoretical investigations.

Recently there has been an increasing activity in research on 
finite elements on surfaces, in particular, we mention the 
following references that are relevant \textcolor{black}{to this} work: finite 
element methods for membrane shell problem based on tangential calculus \cite{HaLa14} (linear) and \cite{HaLaLa14} (nonlinear), 
higher order methods for the Laplace-Beltrami operator \cite{De09}, continuous-discontinuous Galerkin methods for the biharmonic problem \cite{LaLa14}, and the \textcolor{black}{seminal} paper \cite{Dz88} where finite elements for the Laplace-Beltrami was first developed. For 
general background on finite elements for PDEs on surfaces we 
refer to the recent review article \cite{DzEl13} and the references therein.

The outline of the reminder of the paper is as follows: 
In Section 2 we precent the Darcy problem on a surface and 
the necessary background on tangential calculus. In 
Section 3 we define the triangulations and their approximation 
properties, the finite element spaces, the interpolation theory, 
and finally the finite element method. In Section 4 we collect 
necessary results on lifting and extension of functions between 
the exact and discrete surfaces. In Section 5 we derive a priori 
error estimates starting with a Strang \textcolor{black}{lemma} and then estimates 
of the quadrature errors in the forms resulting approximation of 
the geometry, which together with the interpolation results 
yields the final estimate. Estimates for the presssure and tangential
part of the velocity are derived using duality techniques. Finally, 
in Section 6 we present numerical examples.

\section{The Darcy Problem on a Surface}

\subsection{The Surface} Let $\Gamma$ be a closed smooth 
surface embedded in $\IR^3$ with signed distance function $\rho$, exterior unit normal $\bfn = \nabla \rho$, and closest 
point mapping $\bfp:\IR^3 \rightarrow \Gamma$. Then there 
is a $\delta_0>$ such that $\bfp$ maps each point in 
$U_{\delta_0}(\Gamma)$ to precisely one point on $\Gamma$, 
where $U_\delta(\Gamma) = \{ \bfx \in \IR^3 : \rho(\bfx) < \delta\}$ is an open tubular neighborhood of $\Gamma$. 

\subsection{Tangential Calculus} For each function $u$ defined 
on $\Gamma$ we let the extension $u^e$ to the neighborhood 
$U_{\delta_0}(\Gamma)$ be defined by the pull back 
$u^e = u \circ \bfp$. For a function $u: \Gamma \rightarrow \IR$ 
we then define the tangential gradient
\begin{equation}
\textcolor{black}{\nablas u = \bfPs \nabla u^e}
\end{equation}
where $\bfPs = \bfI - \bfn \otimes \bfn$ is the projection 
onto the tangent plane $T_{\bfx}(\Gamma)$. The surface 
divergence of a vector field $\bfu:\Gamma \rightarrow \IR^3$ 
is defined by
\begin{equation}
\divs(\bfu) = \text{tr}(\bfu \otimes \nablas ) 
= \text{div}(\bfu) - \bfn \cdot (\bfu \otimes \nabla)\cdot \bfn 
\end{equation}
Decomposing $\bfu$ into a tangent and normal component
\begin{equation}\label{eq:normaltangentsplit}
\bfu = \bfu_t + u_n \bfn
\end{equation}
we have the identity 
\begin{equation}\label{eq:divnormal}
\divs \bfu = \divs \bfu_t + u_n H
\end{equation}
where 
\begin{equation}
H = \text{tr}(\bfkappa_{\Gamma})
\end{equation}
is twice the mean curvature of the surface and 
$\bfkappa_\Gamma = \bfkappa|_\Gamma$, with 
$\bfkappa = \nabla \otimes \nabla \rho$, is 
the curvature tensor of $\Gamma$. Using Green's 
formula we have 
\begin{equation}\label{eq:Greensdiv}
(\divs \bfv_t, q)_\Gamma = -(\bfv_t, \nablas q)_\Gamma
\end{equation}
for tangential vector fields $\bfv_t$.

\subsection{The Surface Darcy Problem}

\paragraph{Tangential Vector Field Formulation.}
The Darcy problem takes the form: find a tangential vector 
field $\bfu_t:\Gamma \rightarrow T(\Gamma)$ 
representing velocity and the pressure $p:\Gamma\rightarrow \IR$ 
such that
\begin{alignat}{2}\label{eq:Darcya}
\divs \bfu_t &= f \qquad &\text{on $\Gamma$}
\\ \label{eq:Darcyb}
\bfu_t + \nablas p &= \bfg \qquad &\text{on $\Gamma$}
\end{alignat}  
where $f:\Gamma \rightarrow \IR$ is a given function 
such that $\int_\Gamma f = 0$ and \textcolor{black}{$\bfg:\Gamma \rightarrow \IR^3$} 
is a given tangential vector field. The corresponding weak 
form reads: find $(\bfu_t,p) \in \bfV_t \times Q$ such 
that
\begin{equation}\label{eq:darcytan}
a_t((\bfu_t,p),(\bfv_t,q)) = l(q)\qquad 
\forall (\bfv_t,q) \in \bfV_t \times Q
\end{equation}
where 
\begin{align}
a_t((\bfu_t,p),(\bfv_t,q)) &= (\bfu_t,\bfv_t)_\Gamma 
+ (\nablas p, \bfv_t)_\Gamma 
- ( \bfu_t,\nablas q)_\Gamma
\\
l(q) &= (f,q)_\Gamma + (\bfg,\bfv)_\Gamma
\end{align}
\textcolor{black}{with} 
$\bfV_t = \{ \bfv:\Gamma \rightarrow \IR^3 : 
\bfv \in [L^2(\Gamma)]^3, \bfn \cdot \bfv = 0 \}$ 
and $Q = \{ q \in H^1(\Gamma) : \smallint_\Gamma q = 0\}$. 

Since $\Gamma$ is smooth and $p \in Q$ is the solution to the 
elliptic problem $\divs(\nablas p) = \divs \bfu_t - \divs \bfg 
= f  - \divs \bfg$, we have the elliptic regularity estimate 
\begin{equation}
\| p \|_{H^{s+2}(\Gamma)} \lesssim 
\| f - \divs \bfg \|_{H^{s}(\Gamma)} 
\lesssim 
\| f \|_{H^s(\Gamma)}   +  \| \bfg \|_{H^{s+1}(\Gamma)} 
\end{equation}
which combined with 
 $\|\bfu_t\|_{H^{s+1}(\Gamma)}  
 =  \|\bfg - \nablas p\|_{H^{s+1}(\Gamma)}
 \leq  \|\bfg\|_{H^{s+1}(\Gamma)} 
 + \|p\|_{H^{s+2}(\Gamma)}$ gives 
\begin{equation}\label{eq:ellreg}
\| \bfu_t \|_{H^{s+1}(\Gamma)} 
+
\| p \|_{H^{s+2}(\Gamma)} 
\lesssim \|f \|_{H^{s}(\Gamma)}
+
\| \bfg \|_{H^{s+1}(\Gamma)}
\end{equation}

\paragraph{General Vector Field Formulation.}

Letting $\bfu$ be a \textcolor{black}{vector field} with a nonzero normal 
component and recalling the split $\bfu=\bfu_t + u_n \bfn$, see equation (\ref{eq:normaltangentsplit}), we get the problem: find 
$(\bfu,p) \in \bfV \times Q$ such that
\begin{equation}\label{eq:darcygen}
a((\bfu,p),(\bfv,q)) = l(q)\qquad 
\forall (\bfv,q) \in \bfV \times Q
\end{equation}
where
\begin{align}
a((\bfu,p),(\bfv,q))&=(\bfu_t,\bfv_t)_\Gamma + (u_n,v_n)_\Gamma
+(\nablas p,\bfv_t)_{\Gamma} 
-(\bfu_t, \nablas q)_\Gamma
\\
&=(\bfu,\bfv)_\Gamma
-( p,\divs \bfv_t)_{\Gamma} 
+(\divs \bfu_t ,q)_\Gamma
\\
&=(\bfu,\bfv)_\Gamma
-( p,\divs \bfv)_{\Gamma} + (p,H v_n)_\Gamma
+(\divs \bfu ,q)_\Gamma - (u_n H, q)_\Gamma 
\end{align}
and 
$\bfV = \{ \bfv:\Gamma \rightarrow \IR^3 : 
\bfv \in [L^2(\Gamma)]^3\}$ is the space of 
general $L^2(\Gamma)$ vector fields. 

We note that $(\nablas p,\bfv)_{\Gamma}
= (\nablas p,\bfv_t)_{\Gamma}$ , since $\nablas p$ is tangential, 
and therefore we get \textcolor{black}{a} weak enforcement of the tangential condition 
$u_n = \bfu \cdot \bfn = 0$ by setting $q=0$ and 
$\bfv=v_n \bfn$. Testing instead with a tangential vector 
field $\bfv_t\in \bfV_t \subset \bfV$ and $q\in Q$ we recover the Darcy problem (\ref{eq:darcytan}).

\begin{rem} We note that we have the identity
\begin{align}
a((\bfu,p),(\bfv,q))&=(\bfu,\bfv)_\Gamma
+(\nablas p,\bfv_t)_{\Gamma} 
-(\bfu_t, \nablas q)_\Gamma
\\ \label{eq:divform-a}
&=(\bfu,\bfv)_\Gamma
-( p,\divs \bfv_t)_{\Gamma} 
+(\divs \bfu_t ,q)_\Gamma
\\ \label{eq:divform-b}
&=(\bfu,\bfv)_\Gamma
-( p,\divs \bfv)_{\Gamma} + (p,H v_n)_\Gamma
+(\divs \bfu ,q)_\Gamma - (u_n H, q)_\Gamma 
\end{align}
where we used the identity (\ref{eq:divnormal}) for the surface divergence of a 
general vector field in the last step. We note that the third form (\ref{eq:divform-b}) involves quantities that are directly computable while the second form (\ref{eq:divform-a}) involves the \textcolor{black}{surface} divergence of the tangent 
component $\divs \bfv_t$, which is more complicated to compute. When constructing 
a numerical method based on the divergence form (\ref{eq:divform-b}) 
the term  $(v_n, H q)_\Gamma$ either has to be included, which involves 
computation of $H$, or alternatively a stronger penalty on the normal 
component $v_n$ must be added in order to control the inconsistency 
resulting from neglecting the term. Neither alternative is attractive.
%
%
%
%
\end{rem}

\subsection{Masud-Hughes Stabilized Weak Formulation}

The Masud-Hughes weak formulation, 
originally proposed in \cite{MaHu02} for \textcolor{black}{planar} domains, 
for the surface Darcy problem with a general vector field 
velocity takes the form: find $(\bfu,p) \in \bfV \times Q$ 
such that
\begin{equation}\label{eq:mhweak}
A((\bfu,p),(\bfv,q)) = L((\bfv,q)) \qquad 
\forall (\bfv,q) \in \bfV \times Q
\end{equation}
where
\begin{align}\label{eq:mhformcont}
A((\bfu,p),(\bfv,q))
&=(\bfu,\bfv)_\Gamma 
\\ \nonumber 
&\qquad
+(\nablas p,\bfv)_{\Gamma}  
-(\bfu, \nablas q)_\Gamma 
\\ \nonumber 
&\qquad +\frac{1}{2}(\bfu + \nablas p,-\bfv + \nablas q)_\Gamma
\\ \label{eq:mhrighthandsidecont}
L((\bfv,q)) &= (f^e,q)_\Gamma + (\bfg^e,\bfv)_\Gamma
+\frac{1}{2}(\bfg^e,-\bfv + \nablas q)_\Gamma
\end{align}
Expanding the forms we obtain
\begin{align}
A((\bfu,p),(\bfv,q))  
&= 
\frac{1}{2}(\bfu,\bfv)_\Gamma 
+ \frac{1}{2}(\nablas p,\nablas q)_\Gamma
+\frac{1}{2}(\nablas p,\bfv)_{\Gamma}  
-\frac{1}{2}(\bfu, \nablas q)_\Gamma 
\\
L((\bfv,q)) 
&= 
(f^e,q)_\Gamma  
+ \frac{1}{2}(\bfg^e,\bfv + \nablas q)_\Gamma
\end{align}
and thus $A$ has consists of a symmetric and a skew symmetric part.



\section{The Finite Element Method}

\subsection{Triangulation of the Surface}

\paragraph{Parametric Triangulated Surfaces.} 
Let $\widehat{K}\subset \IR^2$ be a reference triangle and 
let $P_{k_g}(\widehat{K})$ be the space of polynomials of 
order less or equal to $k_g$ defined on $\widehat{K}$. Let 
$\Gamma_{h,k_g}$ be a triangulated surface with quasi uniform triangulation $\mcK_{h,k_g}$ and mesh parameter $h\in (0,h_0]$ 
such that each triangle $K=F_{K,k_g}(\widehat{K})$ where 
$F_{K,k_g} \in [P_{k_g}(\widehat{K})]^3$. Let $\bfn_h$ be the elementwise defined normal to $\Gamma_h$. We let 
$\mcE_{h,k_g}$ denote the set of edges in the triangulation. 
For simplicity we use the notation $\mcK_h = \mcK_{h,k_g}$, 
$\mcE_h = \mcE_{h,k_g}$, and $\Gamma_h = \Gamma_{h,k_g}$ 
when appropriate.

\paragraph{Geometry Approximation Property.} We assume that the family 
$\{\Gamma_{h,k_g}, h \in (0,h_0]\}$ approximates 
$\Gamma$ in the following way
\begin{itemize}
\item $\Gamma_{h,k_g} \subset U_{\delta_0}(\Gamma)$ and 
$\bfp:\Gamma_{h,k_g} \rightarrow \Gamma$ is a bijection.
\item The following estimates hold
\begin{equation}\label{eq:geombounds}
\| \bfp \|_{L^\infty(\Gamma_{h,k_g})}\lesssim h^{k_g+1}, 
\qquad
\| \bfn\circ \bfp - \bfn_h \|_{L^\infty(\Gamma_{h,k_g})}\lesssim h^{k_g}
\end{equation}
\end{itemize}
\textcolor{black}{These properties are valid, e.g., if $F_{K,k_g}$ is constructed using Lagrange interpolation of the surface.}
\subsection{Parametric Finite Element Spaces}

Let
\begin{equation}\label{spaceVh}
V_{h,k,k_g} = \{ \bfv: {\bfv\vert_K \circ F_{K,k_g}\in P_{k}(\hat K),\; \forall K\in\mcK_{h,k_g}};\; \bfv\; \in [C^0(\Gamma_h)]^3\}
\end{equation}
be the space of parametric continuous piecewise polynomials 
of order $k$ mapped with a mapping of order $k_g$. We let 
\begin{equation}
\bfV_h = [V_{h,k_u,k_g}]^3, \qquad 
Q_h = \{ q \in V_{h,k_p,k_g}: \smallint_{\Gamma_{h,k_g}} q = 0 \}
\end{equation}
be the finite element spaces for velocity and pressure, 
consisting of continuous piecewise polynomials of order $k_u$ 
and $k_p$, respectively, with parametric map of order $k_g$ (which 
is the same for both spaces). 
%

\subsection{Interpolation}
Let 
\begin{equation}
\pi_{h,1}:L^2(\mcK_{h,1})\ni v \mapsto \textcolor{black}{\pi_{h,1}} v\in  V_{h,k,1}
\end{equation}
be a Scott-Zhang type interpolant. Then, for each element 
$K\in \mcK_{h,1}$ we have the following 
elementwise estimate
\begin{equation}\label{eq:interpollinear}
\| v^e - \textcolor{black}{\pi_{h,1}} v^e \|_{H^m(K)} \lesssim h^{s-m} \| v^e \|_{H^s(N(K))} 
\lesssim \| v \|_{H^s(N^l(K))}
,
\quad m\leq s \leq k+1, \quad m=0,1
\end{equation}
where $N(K)$ is the union of the neighboring elements to 
element $K$ and $N^l(K) = (N(K))^l$. In (\ref{eq:interpollinear}) 
the first inequality 
follows from interpolation theory, see \cite{BrSc08}, and the 
second from the chain rule in combination with $L^\infty$ 
boundedness of derivatives of the closest point map $\bfp$ 
in the tubular neighborhood $U_{\delta_0}(\Gamma)$ which 
follows from smoothness of $\Gamma$.

Next we define the interpolant $\pi_{h,k_g}:L^2(\mcK_h) \rightarrow V_{h,1,k_g}$ as follows
\begin{equation}
\pi_{h,k_g} v^e|_{K} = (\pi_{h,1} v^e) \circ G_{K,k_g,1} 
\end{equation}
where $G_{K,k_g,1} = F_{K,1} \circ F_{K,k_g}^{-1}: K_{k_g} \rightarrow K_1$ is a bijection from the curved triangle $K_{k_g}$ 
to the corresponding flat triangle $K_1$. Using uniform $L^\infty$ bounds on $G_{K,k_g,1}$ and its first order derivative we have the estimates
\begin{align}
\| v^e - \pi_{h,k_g}v^e \|_{H^m(K_{k_g})} 
&\lesssim 
\| v^e - \pi_{h,1} v^e \|_{H^m(K_{1})}
\\
&\quad \lesssim 
h^{s-m} \| v^e \|_{H^s(N(K_1))} 
\lesssim 
h^{s-m} \| v \|_{H^s(N^l(K_1))}
\end{align}
and thus we conclude that we have the estimate
\begin{equation}\label{eq:interpol}
\| v^e - \pi_{h,k_g}v^e \|_{H^m(K_{k_g})} 
\lesssim h^{s-m} \| v \|_{H^s(N^l(K_1))},
\qquad m\leq s \leq k+1, \quad m=0,1
\end{equation}
for all $K\in \mcK_{h,k_g}$. We also have the stability estimate
\begin{equation}\label{eq:interpolstab}
\|  \pi_{h,k_g}v^e \|_{H^m(K_{k_g})} 
\lesssim  \| v \|_{H^m(N^l(K_1))},
\qquad m=0,1
\end{equation}
When appropriate we simplify the notation and write $\pi_h =\pi_{h,k_g}$.

Finally, we have the \textcolor{black}{super--approximation} result
\begin{equation}\label{eq:superapproximation}
\|(I-\pi_h)(\chi^e v)\|_{\Gamma_h} \lesssim 
h \|\chi \|_{W^{k_g+1}_\infty(\Gamma)} \| v \|_{\Gamma_h} 
\end{equation}
for $\chi \in W^{k_g+1}_\infty(\Gamma)$ and $v \in V_{h,k_g}$.


\subsection{Masud-Hughes Stabilized Finite Element Method}

The finite element method based on the 
Masud-Hughes weak formulation (\ref{eq:mhweak}) 
for the surface Darcy problem takes the form:    
find $(\bfu_h,p_h) \in \bfV_h \times Q_h$ such 
that
\begin{equation}\label{eq:fem}
A_h((\bfu_h,p_h),(\bfv,q)) = L_h((\bfv,q))\qquad 
\forall (\bfv,q) \in \bfV_h \times Q_h
\end{equation}
where
\begin{align}\label{eq:mhform}
A_{h}((\bfu,p),(\bfv,q))
&=(\bfu,\bfv)_\Gammah 
\\ \nonumber 
&\qquad
+(\nablash p,\bfv)_{\Gammah}  
-(\bfu, \nablash q)_\Gammah 
\\ \nonumber 
&\qquad +\frac{1}{2}(\bfu + \nablash p,-\bfv + \nablash q)_\Gammah
\\
&=
\frac{1}{2}(\bfu,\bfv)_\Gammah 
+\frac{1}{2}(\nablash p, \nablash q)_\Gammah
+\frac{1}{2}(\nablash p,\bfv)_{\Gammah}  
-\frac{1}{2}(\bfu, \nablash q)_\Gammah 
\\ \label{eq:mhrighthandside}
L_h((\bfv,q)) &= (f^e,q)_\Gammah + (\bfg^e,\bfv)_\Gammah
+\frac{1}{2}(\bfg^e,-\bfv + \nablash q)_\Gammah
\\
&= (f^e,q)_\Gammah +\frac{1}{2}(\bfg^e,\bfv + \nablash q)_\Gammah
\end{align}
 
\begin{rem} \label{remark-cn} We could add the term $c_N (\bfn_h \cdot \bfu,\bfn_h \cdot \bfv)_{\Gammah}$, where $c_N \geq 0$ is a \textcolor{black}{parameter, to} 
enforce the normal constraint more strongly. We will, however, see 
that we can take $c_N=0$, and no significant advantages of taking 
$c_N>0$ has been observed in our numerical experiments.
\end{rem}

\section{Preliminary Results}

\subsection{Extension and Lifting of Functions}

In this section we summarize basic results concerning 
extension and liftings of functions. We refer to 
\cite{BuHaLa14} and \cite{De09} for further details.

\paragraph{Extension.}
Recalling the definition $v^e = v\circ \bfp$ of the 
extension and using the chain rule we obtain the identity 
\begin{equation}\label{eq:tanderext}
\nablash v^e = \bfB^T_t \nablas v  
\end{equation}
where 
\begin{equation}\label{Bmap}
\bfB_t = \bfP_\Gamma (\bfI - \rho \bfkappa)  \bfP_\Gammah : T_{\bfx}(K)\rightarrow
T_{\bfp(\bfx)} (\Gamma)
\end{equation}
and we recall that $\bfkappa = \nabla \otimes \nabla \rho$ which 
may be expressed in the form
\begin{equation}\label{Hform}
\bfkappa(\bfx) = \sum_{i=1}^2 \frac{\kappa_i^e}{1 + \rho(\bfx)\kappa_i^e}\bfa_i^e \otimes \bfa_i^e
\end{equation}
where $\kappa_i$ are the principal curvatures with corresponding
orthonormal principal curvature vectors $\bfa_i$, see \cite{GiTr01} Lemma 14.7. We note that there is $\delta>0$ such that the uniform bound 
\begin{equation}\label{kappalinf}
\textcolor{black}{\|\bfkappa \|_{L^\infty(U_\delta(\Gamma))}\lesssim 1}
\end{equation}
holds.
Furthermore, we have the inverse mapping 
\begin{equation}
\bfB_t^{-1}= \bfP_\Gammah(\bfI-\rho\bfkappa)^{-1} \bfP_\Gamma : T_{\bfp(\bfx)}(\Gamma)\rightarrow T_{\bfx} (K)
\end{equation}

We extend $\bfB_t$ to $T_{\bfx}(K) \oplus N_{\bfx}(K)$, where $N_{\bfx}(K)$ is 
the vector space of vector fields that are normal to $K$ at 
$\bfx \in K$, by defining
\begin{equation}
\bfB = (\bfPs \bfB_t \bfPsh + \bfn \otimes \bfn_h) 
\end{equation}
with inverse 
\begin{equation}
\bfB^{-1} = (\bfPsh \bfB_t^{-1} \bfPs + \bfn_h \otimes \bfn) 
\end{equation}
We note that $\bfB$ and $\bfB^{-1}$ preserves the tangent 
and normal spaces as follows
\begin{equation}\label{BMapPreservesspaces}
\bfB T_{\bfx}(K) = T_{\bfp(\bfx)}(\Gamma), \qquad 
\bfB N_{\bfx}(K) = N_{\bfp(\bfx)}(\Gamma)
\end{equation}
and
\begin{equation}\label{BMapInversePreservesspaces}
\bfB^{-1} T_{\bfp(\bfx)}(\Gamma)=T_{\bfx}(K), \qquad
\bfB^{-1} N_{\bfp(\bfx)}(\Gamma)=N_{\bfx}(K)
\end{equation}
%

For clarity, we will employ the \textcolor{black}{notation}
$\bfB(\bfp(\bfx)) = \bfB(\bfx)$ for each $\bfx \in K$, 
$K \in \mcK_h$, so that we do not have to indicate lift 
or extensions of the operator $\bfB$.

\paragraph{Lifting.}
The lifting $w^l$ of a function $w$ defined on $\Gamma_h$ 
to $\Gamma$ is defined as the push forward
\begin{equation}
(w^l)^e = w^l \circ \textcolor{black}{\bfp} = w \quad \text{on $\Gamma_h$}
\end{equation}
and we have the identity 
\begin{equation}\label{eq:tanderlift}
\nablas w^l = \bfB^{-T} (\nablash w)^l
\end{equation}

\subsection{Estimates Related to $\boldsymbol{B}$}
\label{sec:Bbounds}
Using the uniform bound  \textcolor{black}{(\ref{kappalinf})} 
it follows that
\begin{equation}\label{BBinvbound}
  \| \bfB \|_{L^\infty(\Gamma_h)} \lesssim 1,
 \qquad \| \bfB^{-1} \|_{L^\infty(\Gamma)} \lesssim 1
\end{equation}
Furthermore, we have the estimates 
\begin{equation}\label{BBTbound}
\|\bfB - \bfI\|_{L^\infty(\Gamma_h)} 
\lesssim h^{k_g} 
 \qquad \| \bfI - \bfB^T \bfB \|_{L^\infty(\Gamma_h)} 
 \lesssim h^{k_g+1}
\end{equation}
To prove the first estimate in (\ref{BBTbound}) we note, 
using the definition (\ref{Bmap}) of $\bfB$ and the bound (\ref{eq:geombounds}) on $\rho$ that 
\begin{equation}\label{Bsimple}
\bfB = \bfPs \bfPsh + \bfn \otimes \bfn_h + O(h^{k_g+1})
\end{equation}
Next writing $\bfI = \bfPs + \bfn \otimes \bfn = 
\bfPs^2 + \bfn \otimes \bfn$, where we used that 
$\bfPs$ is a projection, we obtain
\begin{equation}\label{eq:BminusI}
\bfB - \bfI = \bfPs(\bfPsh - \bfPs) 
+ \bfn \otimes (\bfn_h - \bfn) + O(h^{k_g+1})
\end{equation}
and thus the estimate follows using (\ref{eq:geombounds}) 
since $\| \bfPs - \bfPsh \|_{L^\infty(\Gammah)} 
\lesssim \| \bfn - \bfn_h \|_{L^\infty(\Gammah)} 
\lesssim h^{k_g}$.

For the second estimate in (\ref{BBTbound}) we use the 
identities $\bfI = \bfPsh + \bfn_h \otimes \bfn_h$ and 
(\ref{Bsimple}) to conclude that 
\begin{equation}
\bfB^T \bfB - \bfI 
= \bfPsh \bfPs \bfPsh + \bfn_h \otimes \bfn_h - \bfI + O(h^{k_g+1})
= \bfPsh \bfPs \bfPsh - \bfPsh + O(h^{k_g+1})
\end{equation}
Now  
\begin{equation}
\bfPsh \bfPs \bfPsh - \bfPsh = (\bfPsh \bfn) \otimes (\bfPsh \bfn)
= (\bfPsh (\bfn- \bfn_h)) \otimes (\bfPsh (\bfn - \bfn_h))
\sim O(h^{2k_g})
\end{equation}
where we used the bound for the error in the discrete normal 
(\ref{eq:geombounds}). Thus the second bound in 
(\ref{BBTbound}) follows.

\textcolor{black}{Further,} the surface measure $d \Gamma = |\bfB| d \Gammah$,
where $|\bfB| =| \text{det}(\bfB)|$ is the absolute value 
of the determinant of $\bfB$ and we have the following 
estimates
\begin{equation}\label{detBbound}
\| 1 - |\bfB| \|_{L^\infty(\Gamma_h)} \lesssim h^{k_g+1}, 
\quad \||\bfB|\|_{L^\infty(\Gamma_h)} \lesssim 1, \quad \||\bfB|^{-1}\|_{L^\infty(\Gamma_h)} \lesssim 1
\end{equation}

\subsection{Norm Equivalences}

In view of the bounds in Section \ref{sec:Bbounds} and the 
identities (\ref{eq:tanderext}) and (\ref{eq:tanderlift}) 
we obtain the following equivalences
\begin{equation}\label{eq:normequ}
\| v^l \|_{L^2(\Gamma)}
\sim \| v \|_{L^2(\Gammah)}, \qquad \| v \|_{L^2(\Gamma)} \sim
\| v^e \|_{L^2(\Gammah)}
\end{equation}
and
\begin{equation}\label{eq:normequgrad}
\| \nabla_\Gamma v^l \|_{L^2(\Gamma)} \sim \| \nablash v \|_{L^2(\Gammah)},
\qquad \| \nablas v \|_{L^2(\Gamma)} \sim
\| \nablash v^e \|_{L^2(\Gammah)}
\end{equation}

\subsection{Poincar\'e Inequality}

We have the  following Poincar\'e inequality 
\begin{equation}\label{PoincareGammah}
\| v \|_{\Gammah} 
\lesssim 
\| \nablash v \|_{\Gammah} 
\quad \forall v \in Q_h
\end{equation}
To prove (\ref{PoincareGammah}) we let 
\begin{equation}\label{eq:average}
\lambda_S(v) = |S|^{-1}  \int_S v, \qquad S \in \{\Gamma,\Gammah\}
\end{equation}
be the average over $S$. Using the fact that $\alpha = 
\lambda_\Gammah(v)$ is 
the constant that minimizes 
$\| v - \alpha \|_{\Gammah}$, norm equivalence (\ref{eq:normequ}) to pass from $\Gammah $ to $\Gamma$,  the standard  Poincar\'e estimate on $\Gamma$, and at last norm equivalence (\ref{eq:normequgrad}) to pass back to $\Gammah$, we obtain 
\begin{equation}
\| v - \lambda_\Gammah(v) \|_{\Gammah} 
\leq \|v - \lambda_\Gamma(v^l) \|_{\Gammah} 
\lesssim \|v^l -\lambda_\Gamma(v^l)\|_\Gamma 
\lesssim \| \nablas v^l \|_{\Gamma}
\lesssim \| \nablash v \|_{\Gammah}
\end{equation}
which proves (\ref{PoincareGammah}).
\section{Error Estimates}

\subsection{Norms}
Let 
\begin{equation}
\tn (\bfv,q)\tn^2 = \|\bfv \|^2_\Gamma + \| \nablas q \|^2_\Gamma, 
\qquad 
\tn (\bfv,q)\tn_h^2 = \|\bfv \|^2_\Gammah + \| \nablash q \|^2_\Gammah
\end{equation}
Using (\ref{eq:normequ}) and (\ref{eq:normequgrad}) we have the following equivalences 
\begin{equation}
\tn (\bfv^l,q^l) \tn \sim \tn (\bfv,q) \tn_h, 
\qquad 
\tn (\bfv,q) \tn \sim \tn (\bfv^e,q^e) \tn_h
\end{equation}

\subsection{Coercivity and Continuity}
\begin{lem}\label{lem:coercont} The following \textcolor{black}{statements} hold:
\begin{itemize}
\item The form $A_h$ is coercive 
and continuous
\begin{equation}\label{eq:coer}
\tn (\bfv,q ) \tn_h^2 \lesssim A_h((\bfv,q ),(\bfv,q ))
\qquad \forall (\bfv,q) \in \bfV_h \times Q_h
\end{equation}
\begin{equation}\label{eq:cont}
A_h((\bfv,q ),(\bfw,r))
\lesssim \tn (\bfv,q )\tn_h \tn (\bfw,r)\tn_h
\qquad \forall(\bfv,q), (\bfw,r) \in \bfV_h \times Q_h
\end{equation}
for all $h\in (0,h_0]$. 
\item The form  $L_h$ is continuous
\begin{equation}\label{eq:Lhcont}
L_h((\bfv,q)) \lesssim (\|f\|_\Gamma + \| \bfg \|_\Gamma ) 
\tn (\bfv,q) \tn_h \qquad \forall(\bfv,q) \in \bfV_h \times Q_h
\end{equation}
for all $h\in (0,h_0]$. 

\item There exists a  unique solution to (\ref{eq:fem}).
\end{itemize}
\end{lem}
\begin{proof} Coercivity of $A_h$ follows directly from the 
definition (\ref{eq:mhform}) of the stabilized bilinear 
form. Continuity of $A_h$ follows from the Cauchy-Schwarz 
inequality.  Continuity of $L_h$ follows directly from the Cauchy-Schwarz inequality and the Poincar\'e inequality. Existence and uniqueness of a solution to (\ref{eq:fem}) follows from the 
Lax-Milgram lemma.
\end{proof}

\begin{rem} Clearly the analogous results holds for the 
continuous forms $A$ and $L$ on $\bfV\times Q$, defined in (\ref{eq:mhformcont}) and (\ref{eq:mhrighthandsidecont}), and 
the variational problem (\ref{eq:mhweak}).
\end{rem}

\subsection{Discrete Stability Estimate}

\begin{lem}\label{lem:stab} The solution $(\bfu_h,p_h)$ to (\ref{eq:fem}), satisfies the stability estimate
\begin{equation}\label{eq:staba}
\tn (\bfu_h,p_h) \tn_h 
+ h^{-1} \| \bfn \cdot \bfu_h \|_\Gammah
\lesssim \|f \|_{\Gamma} + \| \bfg \|_\Gamma 
\end{equation}
\end{lem}
\begin{proof}  {\bf First Test Function.} With $(\bfv,q) = (\bfu_h,p_h)$ in (\ref{eq:fem}) 
we obtain
\begin{equation}
\tn (\bfu_h, p_h)\tn_h^2 
\lesssim A_h((\bfu_h, p_h),(\bfu_h, p_h))
= L_h((\bfu_h, p_h)) 
\lesssim  (\|f \|_{\Gamma} + \| \bfg \|_\Gamma)  \tn (\bfu_h, p_h) \tn_h
\end{equation} 
where we used (\ref{eq:coer}) and (\ref{eq:cont}) together with 
the continuity 
\begin{equation}
L_h((\bfv,q)) \lesssim 
( \|f \|_{\Gamma} + \| \bfg \|_\Gamma)  \tn (\bfv,q)\tn_h, \forall 
(\bfv,q) \in \bfV_h \times Q_h
\end{equation}
of $L_h$, which follows 
from Cauchy-Schwarz and the Poincar\'e inequality 
(\ref{PoincareGammah}). Thus we conclude that 
\begin{equation}\label{eq:energystab}
\tn (\bfu_h, p_h)\tn_h\lesssim \|f \|_{\Gamma} + \| \bfg \|_\Gamma
\end{equation}

\paragraph{Second Test Function.}
Setting $q=0$ in (\ref{eq:fem}) 
we note that the following equation holds
\begin{equation}
(\bfu_h,\bfv)_{\Gamma_h} + (\nablash p_h,\bfv)_{\Gamma_h} = 
(\bfg^e,\bfv)_{\Gamma_h}, \qquad \forall \bfv\in \bfV_h
\end{equation}
Choosing the test function $\bfv =\pi_h (\bfn \pi_h(\bfn \cdot \bfu_h))$ we get the identity
\begin{equation}
\underbrace{(\bfu_h,\pi_h(\bfn \pi_h (\bfn \cdot \bfu_h)))_{\Gamma_h}}_{I} 
+
\underbrace{(\nablash p_h, \pi_h(\bfn \pi_h (\bfn\cdot\bfu_h)) )_{\Gamma_h}}_{II} 
=
\underbrace{(\bfg^e,\pi_h(\bfn \pi_h (\bfn \cdot \bfu_h)))_{\Gamma_h}}_{III}
\end{equation} 

\paragraph{Term $\bfI$.} We have
\begin{align}
\nonumber
&(\bfu_h,\pi_h(\bfn \pi_h (\bfn\cdot \bfu_h)))_{\Gamma_h} 
\\
&\qquad=
(\bfn \cdot \bfu_h , \bfn \cdot \bfu_h)_{\Gamma_h} 
+ (\pi_h (\bfn \pi_h (\bfn \cdot \bfu_h)) - \bfn (\bfn \cdot \bfu_h) ,
 \bfn \cdot \bfu_h)_{\Gamma_h} 
\\
&\qquad \geq
\| \bfn \cdot \bfu_h\|^2_\Gammah
- 
\|\pi_h (\bfn \pi_h (\bfn \cdot \bfu_h)) - \bfn (\bfn \cdot \bfu_h)\|_\Gammah
\| \bfn \cdot \bfu_h\|^2_\Gammah
\\
&\qquad \geq
( 1 - \delta) \| \bfn \cdot \bfu_h\|^2_\Gammah
- 
\delta^{-1} \underbrace{\|\pi_h (\bfn \pi_h (\bfn \cdot \bfu_h)) - \bfn (\bfn \cdot \bfu_h)\|_\Gammah^2}_{\bigstar}
\end{align}
where we used the Cauchy-Schwarz inequality and finally 
 $ab \leq \delta a^2 + \delta^{-1} b^2, \delta>0$.
Estimating the second term on the right hand side by adding and subtracting suitable terms we obtain
\begin{align}
\bigstar& \lesssim 
\|(\pi_h -I) (\bfn \pi_h (\bfn \cdot \bfu_h)) \|^2_\Gammah
+
\|\bfn (\pi_h - I)  (\bfn \cdot \bfu_h) \|^2_\Gammah
\\
&\lesssim h^2\| \pi_h (\bfn \cdot \bfu_h) \|^2_\Gammah
+
h^2 \| \bfu_h \|_\Gammah^2
\\
&\lesssim h^2 \| \bfu_h \|^2_\Gammah
\end{align}
where we used \textcolor{black}{the super--approximation} (\ref{eq:superapproximation}) 
 and in the last step the $L^2$ stability (\ref{eq:interpolstab}) of the interpolant $\pi_h$. We thus arrive at 
\begin{equation}\label{stab:aaI}
I \gtrsim 
(1-\delta) \| \bfn \cdot \bfu_h \|^2_\Gammah  
- 
C h^2 \| \bfu_h\|^2_\Gammah
\end{equation}
%
%

\paragraph{Term $\bfI\bfI$.} We have 
\begin{align}
II &\lesssim 
\|\nablash p_h\|_{\Gamma_h} 
\|\bfPsh \pi_h(\bfn \pi_h (\bfn \cdot\bfu_h)) \|_{\Gamma_h}
\\
&
\lesssim h \|\nablash p_h\|_{\Gamma_h} 
\|\bfn \cdot\bfu_h \|_{\Gamma_h}
\\ \label{stab:aaII}
&\leq  \delta^{-1} C h^2 \| \nablash p_h \|^2_{\Gamma_h}
+\delta \|\bfn \cdot\bfu_h\|^2_{\Gamma_h}
\end{align}
Here we used the estimate 
\begin{align}
\nonumber
&\|\bfPsh \pi_h(\bfn \pi_h (\bfn \cdot\bfu_h)) \|_{\Gamma_h}
\\
&\qquad \lesssim 
\|\bfPsh (\pi_h-I)(\bfn \pi_h (\bfn \cdot\bfu_h)) \|_{\Gamma_h}
+ \|(\bfPsh \bfn)\pi_h (\bfn \cdot\bfu_h )\|_{\Gamma_h}
\\
&\qquad \lesssim 
h \|\pi_h (\bfn \cdot\bfu_h) \|_{\Gamma_h}
+ \|\bfPsh \bfn\|_{L^\infty(\Gamma_h)} 
\|\pi_h (\bfn \cdot\bfu_h) \|_{\Gamma_h}
\\
&\qquad \lesssim 
h \|\bfn \cdot\bfu_h \|_{\Gamma_h} 
+ h^{k_g} \|\bfn \cdot\bfu_h \|_{\Gamma_h}
\\
&\qquad \lesssim h \|\bfn \cdot\bfu_h \|_{\Gamma_h}
\end{align}
where we added and subtracted $ \bfn \pi_h (\bfn \cdot\bfu_h )$, 
used the triangle inequality \textcolor{black}{followed by the super--approximation 
(\ref{eq:superapproximation}) and the}
$L^2$ stability (\ref{eq:interpolstab}) of the interpolant, \textcolor{black}{and finally 
used} the fact that $k_g \geq 1$ and $h\in (0,h_0]$.

\paragraph{Term $\bfI\bfI\bfI$.} We have
\begin{align}
|III|&=|(\bfg^e,\pi_h(\bfn \pi_h (\bfn\cdot \bfu_h)))_{\Gamma_h}|
\\
&=
|(\bfg^e,(\pi_h - I)(\bfn \pi_h (\bfn\cdot \bfu_h)))_{\Gamma_h}|
\\
&\lesssim 
\|\bfg^e\|_{\Gamma_h} \|(\pi_h - I)(\bfn \pi_h (\bfn\cdot \bfu_h))\|_{\Gamma_h}
\\
&\lesssim 
h \|\bfg^e\|_{\Gamma_h} \| \pi_h (\bfn \cdot \bfu_h) \|_{\Gamma_h}
\\
&\lesssim 
h \|\bfg^e\|_{\Gamma_h} \| \bfn \cdot \bfu_h \|_{\Gamma_h}
\\ \label{stab:aaIII}
&\leq 
\delta^{-1} C h^2  \|\bfg^e\|_{\Gamma_h}^2 
+
\delta \| \bfn\cdot \bfu_h \|^2_{\Gamma_h}
\end{align}
where we used the fact that $(\bfg^e,\bfn)_\Gamma = 0$ to 
subtract $\bfn \pi_h (\bfn\cdot \bfu_h)$, used the Cauchy-Schwarz 
inequality,  used \textcolor{black}{the super--approximation} (\ref{eq:superapproximation}), 
and the $L^2$ stability (\ref{eq:interpolstab}) of the interpolant.

\paragraph{Conclusion.}
Collecting the bounds (\ref{stab:aaI}), (\ref{stab:aaII}), 
and (\ref{stab:aaIII}), for $I$, $II$, and $III$, respectively, 
we arrive at
\begin{align}
\left(1 - 3 \delta \right)
\| \bfn \cdot \bfu_h \|^2_{\Gamma_h} 
&\lesssim
\delta^{-1} h^2 \|\bfg^e \|_{\Gamma_h} 
+ \delta^{-1} h^2 \| \bfu_h \|^2_{\Gamma_h} 
+ \delta^{-1} h^2 \|\nablash p_h\|^2_{\Gamma_h}
\\
&\lesssim
\delta^{-1} h^2 \left( \|\bfg^e\|_{\Gamma_h} 
+ \tn (\bfu_h,p_h) \tn^2_{h} \right)
\\
&\lesssim
\delta^{-1} h^2 \left( \|f \|_\Gamma + \|\bfg\|_{\Gamma}  \right) 
\end{align}
where we used (\ref{eq:energystab}) in the last step. Choosing 
$\delta$ small enough completes the proof.
\end{proof}

\subsection{Interpolation Error Estimates}
Using the interpolation error estimate (\ref{eq:interpol}) we 
directly obtain the following interpolation estimates in the energy norm
\begin{align}
\tn (\bfv, q) - (\pi_h\bfv^e,\pi_h q^e)^l \tn 
&\sim 
\tn (\bfv, q) - (\pi_h\bfv^e,\pi_h q^e)\tn_h 
\\ \label{eq:interpol-energy}
& \lesssim 
h^{k_u + 1} \|\bfv\|_{H^{k_u+1}(\Gamma)}
+ 
h^{k_p} \| q \|_{H^{k_p+1}(\Gamma)} 
\end{align}
for $(\bfv,q) \in H^{k_u+1}(\Gamma) \times H^{k_p+1}(\Gamma).$
If $\bfv$ is tangential, $\bfn \cdot \bfv = 0$, we also have the estimate
\begin{equation}\label{eq:interpol-normal}
\|\bfn \cdot \pi_h \bfv \|_\Gammah 
= 
 \|\bfn \cdot (\pi_h \bfv - \bfv) \|_\Gammah
\lesssim 
h \|\bfv \|_{H^1(\Gamma)}  
\end{equation}

\subsection{Strang's Lemma}

\begin{lem}\label{lem:strang}  Let $(\bfu,p)$ be the 
solution to (\ref{eq:darcytan}) and $(\bfu_h,p_h)$ the 
solution to (\ref{eq:fem}), then the following estimate 
holds 
\begin{align}\label{eq:strang}
\tn (\bfu,p) - (\bfu_h,p_h)^l \tn
&\lesssim \tn (\bfu^e - \pi_h \bfu^e,p^e - \pi_h p^e) \tn_h
\\ 
\nonumber 
& \qquad 
+\frac{A((\pi_h \bfu^e ,\pi_h p^e)^l,(\bfv,q)^l) 
- A_h(\pi_h \bfu^e,\pi_h p^e),(\bfv,q)))}{\tn  (\bfv,q) \tn_h}
\\ 
\nonumber
& \qquad 
+ \frac{L(\bfv,q)^l) 
- L_h((\bfv,q))}{\tn (\bfv,q) \tn_h}
\end{align} 
for all $h\in (0,h_0]$.
%
\end{lem}
\begin{proof} 
Adding and subtracting an interpolant and 
using the triangle inequality we obtain
\begin{align}
 \tn (\bfu,p) - (\bfu_h^l,p_h^l) \tn
 &\sim 
 \tn (\bfu^e,p^e) - (\bfu_h,p_h) \tn_h
 \\ 
 &\lesssim 
 \tn (\bfu^e,p^e) - (\pi_h \bfu^e,\pi_h p^e) \tn_h 
 + \tn (\pi_h \bfu^e,\pi_h p^e) - (\bfu_h,p_h) \tn_h
\end{align}
To estimate the second term use (\ref{eq:coer}) to 
conclude that
\begin{align}
\label{eq:infsup}
\tn (\pi_h \bfu^e, \pi_h p^e) - (\bfu_h,p_h) \tn_h
&\lesssim 
\sup_{(\bfv,q)\in \bfV_h\times Q_h} \frac{A_h( (\pi_h \bfu^e, \pi_h p^e) - (\bfu_h,p_h),(\bfv,q))}{\tn (\bfv,q)\tn_h}
\end{align}
Using Galerkin orthogonality (\ref{eq:fem}) to eliminate $(\bfu_h,p_h)$, and then adding the weak form of the exact 
problem, the numerator may be written 
in the following form
\begin{align}\nonumber
&A_h( (\pi_h \bfu^e,\pi_h p^e)  - (\bfu_h,u_h),(\bfv,q)) 
\\
&\qquad = A_h( (\pi_h \bfu^e,\pi_h p^e),(\bfv,q)) - L_h((\bfv,q))
\\
&\qquad = A_h( (\pi_h \bfu^e,\pi_h p^e),(\bfv,q)) 
\underbrace{- A((\bfu,p),(\bfv,q)^l) + L((\bfv,q)^l)}_{=0}
- L_h((\bfv,q))
\\
&\qquad =A_h( (\pi_h \bfu^e,\pi_h p^e),(\bfv,q)) 
- \textcolor{black}{A( (\pi_h \bfu^e,\pi_h p^e)^l,(\bfv,q)^l)}
\\ \nonumber
&\qquad \qquad 
+ L((\bfv,q)^l) - L_h((\bfv,q))
\\ \nonumber
&\qquad \qquad + A( ((\pi_h \bfu^e)^l - u,(\pi_h p^e)^l - p),(\bfv,q)^l)
\end{align}
where at last we added and subtracted an interpolant. Estimating 
the right hand side and using (\ref{eq:infsup}) the lemma follows immediately.
\end{proof}

\subsection{Quadrature Error Estimates}

\begin{lem}\label{lem:quad} The following estimates hold 
\begin{align}
\label{eq:quada}
|(\bfv^l,\bfw^l)_{\mcK_h^l} - (\bfv,\bfw)_{\mcK_h}| 
&\lesssim h^{k_g+1}\|\bfv\|_{\mcK_h} \|\bfw\|_{\mcK_h} 
\\ \label{eq:quadb}
|(\bfv^l,\nablas q^l)_{\mcK_h^l} 
- (\bfv,\nablash q)_{\mcK_h}|  
&\lesssim h^{k_g+1} \big( \|\bfv\|_{\mcK_h} 
+ h^{-1} \|\bfn \cdot \bfv\|_{\mcK_h} \big)
\|\nablash q \|_{\mcK_h} 
%
\\ \label{eq:quadc}
|(\nablas q^l, \nablas r^l)_{\mcK_h^l} 
- (\nablash q, \nablash r)_{\mcK_h}|
&\lesssim h^{k_g+1} \|\nablash q \|_{\mcK_h^l} 
\|\nablash r \|_{\mcK_h}
\\ \label{eq:quadl}
|L((v^L,q^l)) - L_h((v,q))| 
&\lesssim h^{k_g+1} (\|f \|_{\Gamma} + \|\bfg\|_{\Gamma} ) 
\tn (v,q) \tn_h
\end{align}
for all $(\bfv,q),(\bfw,r) \in \bfV_h \times Q_h$ and $h\in (0,h_0]$.
\end{lem}
\begin{proof} {\bf (\ref{eq:quada}):} Changing domain of 
integration from $\mcK_h^l$ to $\mcK_h$ in the first term 
and using the bound (\ref{detBbound}) for $|\bfB|$ we obtain
\begin{align}
|(\bfv^l,\bfw^l)_{\mcK_h^l} - (\bfv,\bfw)_{\mcK_h}|
&=|((|\bfB|-1)\bfv,\bfw)_{\mcK_h}|
\\
&\lesssim \| |\bfB| - 1 \|_{L^\infty(\mcK_h)} 
\|\bfv \|_{\mcK_h}  \| \bfw \|_{\mcK_h} 
\\
&\lesssim h^{k_g+1} \| \bfv \|_{\mcK_h}  \| \bfw \|_{\mcK_h} 
\end{align}

\noindent{\bf (\ref{eq:quadb}):} Changing domain of integration 
from $\mcK_h^l$ to $\mcK_h$ in the first term we obtain
\begin{align}
\nonumber
&|(\bfv^l,\nablas q^l)_{\mcK_h^l} 
- (\bfv,\nablash q)_{\mcK_h}| 
\\
&\qquad = |( \bfv^l, \bfB^{-T} (\nablash q)^l)_{\mcK_h^l} 
- (\bfv,\nablash q)_{\mcK_h}| 
\\
&\qquad = |( (\bfPsh (|\bfB| \bfB^{-1} -\bfI) \bfv , 
\nablash q)_{\mcK_h}|
\end{align}
Here we have the identity 
\begin{align}
\bfPsh (|\bfB| \bfB^{-1} - \bfI) \bfv 
&= \bfPsh (|\bfB|-1) \bfB^{-1} +
\bfPsh (\bfB^{-1} - \bfI)  
\end{align}
where we note that the first term i $O(h^{k_g+1})$ using (\ref{detBbound}) and the second takes the form 
\begin{align}
\bfPsh (\bfB^{-1} - \bfI) 
&= 
\bfPsh (\bfPsh \bfPs + \bfn_h \otimes \bfn + O(h^{k_g+1}) - \bfI)  
\\
&=\bfPsh \bfQs + O(h^{k_g+1}) 
\end{align}
Thus we conclude that
\begin{align}
\nonumber
|(\bfv^l,\nablas q^l)_{\mcK_h^l} 
- (\bfv,\nablash q)_{\mcK_h}|
&=|( (\bfPsh (|\bfB| \bfB^{-1} -\bfI) \bfv , 
\nablash q)_{\mcK_h}|
\\
&\lesssim h^{k_g+1} \|\bfv \|_{\mcK_h} \|\nablash q \|_{\mcK_h}
+ |( (\bfPsh \bfn) (\bfn \cdot \bfv),\nablash q)_{\mcK_h} | 
\\
&\lesssim h^{k_g+1} \|\bfv \|_{\mcK_h} \|\nablash q \|_{\mcK_h}
+ h^{k_g} \|\bfn \cdot \bfv\|_{\mcK_h} \|\nablash q\|_{\mcK_h}  
\end{align}

\noindent{\bf (\ref{eq:quadc}):} Using (\ref{eq:tanderlift}), changing domain of integration from $\mcK_h^l$ to $\mcK_h$ 
in the first term we obtain
\begin{align}
\nonumber
&|(\nablas q^l,\nablas r^l)_{\mcK_h^l} 
- (\nablash q,\nablash r)_{\mcK_h}|
\\
&\qquad 
=
|(\bfB^{-T} (\nablash q)^l, \bfB^{-T} (\nablash r)^l)_{\mcK_h^l} 
- (\nablash q,\nablash r)_{\mcK_h}|
\\
&\qquad 
=|( (|\bfB|\bfB^{-1} \bfB^{-T} - \bfI)\nablash q,\nablash r)_{\mcK_h}|
\\
&\qquad 
\lesssim  \| (|\bfB| \bfB^{-1}\bfB^{-T} - \bfI) \|_{L^\infty(\mcK_h)} 
\| \nablash q \|_{\mcK_h}  \| \nablash r \|_{\mcK_h} 
\\
&\qquad 
\lesssim h^{k_g+1} \| \nablash q \|_{\mcK_h}  \| \nablash r \|_{\mcK_h} 
\end{align}
Here we used the estimate
\begin{align}
& \| (|\bfB| \bfB^{-1}\bfB^{-T} - \bfI) \|_{L^\infty(\mcK_h)} 
\lesssim 
 \| |\bfB| - 1\|_{L^\infty(\mcK_h)} 
 \| \bfB^{-1}\|_{L^\infty(\mcK_h)} \|\bfB^{-T}\|_{L^\infty(\mcK_h)}
\\
&\qquad  +  \| \bfB^{-1}\|_{L^\infty(\mcK_h)} 
 \| \bfI - \bfB \bfB^T \|_{L^\infty(\mcK_h)} 
 \|\bfB^{-T}\|_{L^\infty(\mcK_h)} \lesssim h^{k_g+1}
\end{align}
where we employed (\ref{BBinvbound}), (\ref{BBTbound}), and (\ref{detBbound}).

\noindent{\bf (\ref{eq:quadl}):} Changing domain of integration 
from $\mcK_h^l$ to $\mcK_h$ and using (\ref{detBbound}) we obtain
\begin{align}
&|(f,q^l)_{\mcK_h^l} - (f^e,q)_{\mcK_h}| 
+ |(\bfg,\bfv^l)_{\mcK_h} - (\bfg^e,\bfv)_{\mcK_h}|
\\
&\qquad = |((|\bfB|-1)f^e,q )_{\mcK_h}| 
+ |((|\bfB|-1) \bfg^e,\bfv)_{\mcK_h}|
\\
&\qquad \lesssim  \| |\bfB|-1\|_{L^\infty(\Gamma_h)} ( \|f^e\|_{\mcK_h}
\|q\|_{\mcK_h} + \|\bfg^e\|_{\mcK_h} \|\bfv\|_{\mcK_h} )
\\
&\qquad  \lesssim \| |\bfB|-1\|_{L^\infty(\Gamma_h)} 
( \|f^e\|^2_{\mcK_h} +  \|\bfg^e\|^2_{\mcK_h})^{1/2}
( \|\nablash q\|^2_{\mcK_h} + \|\bfv\|^2_{\mcK_h} )^{1/2}
\\
&\qquad \lesssim h^{k_g+1} (\|f\|_{\Gamma} + \|\bfg \|_\Gamma ) 
\tn (\bfv,q) \tn_h
\end{align}
where we used the Poincar\'e inequality (\ref{PoincareGammah}).
\end{proof}
We collect our results in a convenient form for the developments 
below in the following corollary.
\begin{cor}\label{cor:quadest} The following estimates hold
\begin{align}\nonumber
&|A((\bfv,q)^l,(\bfw,r)^l)
- A_h((\bfv,q), (\bfw,r))|
\\ \label{eq:quadestAh}
&\qquad \lesssim h^{k_g+1}
\Big(\tn (\bfv,q)\tn_h + h^{-1}\|\bfn \cdot \bfv \|^2_\Gammah \Big)
\Big(\tn (\bfw,r)\tn_h + h^{-1}\|\bfn \cdot \bfw \|^2_\Gammah \Big)
\end{align}
and
\begin{equation}\label{eq:quadestLh}
|L((\bfv,q)^l) - L_h((\bfv,q))|
\lesssim h^{k_g+1}\Big(\|f\|_{\Gamma} + \| \bfg \|_{\Gamma} \Big) 
 \tn (\bfv,q) \tn_h 
\end{equation}
for all $(\bfv,q)$ and $(\bfw,r)  \in \bfV_h \times Q_h$ and $h \in (0,h_0]$.
\end{cor}

\subsection{Error Estimates}

\begin{thm}\label{thm:energyest} Let $(\bfu,p)$ be the 
solution to (\ref{eq:darcytan}) and $(\bfu_h,p_h)$ 
the solution to (\ref{eq:fem}) and assume that 
the geometry approximation property holds, then for 
the following estimate 
holds
\begin{align}\label{eq:energyest}
\tn (\bfu - \bfu_{h}^l,  p - p_h^l ) \tn
&\lesssim h^{k_u+1}\| \bfu \|_{H^{k_u+1}(\Gamma)} 
+ h^{k_p} \| p \|_{H^{k_p+1}(\Gamma)}
+ h^{k_g} (\|f\|_\Gamma + \|\bfg\|_{\Gamma} )
\end{align}
for all $h\in (0,h_0]$.
\end{thm}
\begin{proof} Starting from Strang's lemma we need to estimate the 
three terms on the right hand side in (\ref{eq:strang}).  For the first 
term using the interpolation estimate gives
\begin{equation}
\tn \bfu -\pi_h \bfu^e, p - \pi_h p^e \tn 
\lesssim 
h^{k_u+1}\| u \|_{H^{k_u+1}(\Gamma)} 
+ 
h^{k_p} \| p \|_{H^{k_p+1}(\Gamma)}
\end{equation}
For the second term  using the quadrature estimate (\ref{eq:quadestAh}) we obtain
\begin{align}
\nonumber
& \Big|A((\pi_h \bfu^e ,\pi_h p^e)^l,(\bfv,q)^l) 
- A_h(\pi_h \bfu^e,\pi_h p^e),(\bfv,q)))\Big|
\\
&\qquad 
\lesssim h^{k_g+1} 
\underbrace{\Big( \tn (\pi_h \bfu^e ,\pi_h p^e) \tn_h 
+ h^{-1} \|\bfn \cdot \pi_h \bfu^e \|_\Gammah \Big)}_{
\lesssim \tn (\bfu,p)\tn +  \| \bfu \|_{H^1(\Gamma)}
}
 \underbrace{\Big( \tn (\bfv ,q) \tn_h 
+ h^{-1} \|\bfn \cdot \bfv \|_\Gammah \Big)}_{\lesssim h^{-1} \tn (\bfv ,q) \tn_h}
\\
&\qquad 
\lesssim h^{k_g} (  \|\bfu\|_{H^1(\Gamma)} + \|p\|_{H^1(\Gamma)}  ) 
 \tn (\bfv ,q) \tn_h 
\end{align}
where we used the interpolation estimate (\ref{eq:interpol-normal}) 
for the first term. For the third 
term, applying (\ref{eq:quadestLh}), directly gives 
\begin{equation}
\Big| L((\bfv,q)^l) -  L_h(\bfv,q) \Big| \lesssim h^{k_g+ 1} 
\Big(\|f\|_\Gamma + \|\bfg\|_\Gamma \Big) \tn (\bfv ,q) \tn_h 
\end{equation}
Combining the three estimates with the Strang lemma we directly 
obtain the desired estimate.
%
%
%
%
\end{proof}


\begin{thm}\label{thm:L2estp} Under the same assumptions 
as in Theorem \ref{thm:energyest} the following estimate 
holds
\begin{align}\nonumber
 \| p - p_h^l \|_\Gamma
&\lesssim h^{k_u+2}\| \bfu \|_{H^{k_u+1}(\Gamma)} 
+ h^{k_p+1} \| p \|_{H^{k_p+1}(\Gamma)}
+ h^{k_g+1} (\|f\|_\Gamma + \|\bfg\|_\Gamma )
\end{align}
for all $h\in (0,h_0]$. 
\end{thm}
\begin{proof} Recall that $\lambda_S(v) = |S|^{-1}\int_S v$ is 
the average of a function in $L^2(S)$, $S\in \{\Gamma, \Gammah\}$, see (\ref{eq:average}). Then we have
\begin{equation}
\|p  -p_h^l\|_\Gamma 
\leq  \underbrace{\|p - (p_h^l - \lambda_\Gamma(p_h^l) )\|_\Gamma}_{I} 
+ \underbrace{| \lambda_\Gamma(p_h^l) -  \lambda_\Gammah(p_h)|}_{II}
\end{equation}
where we added and subtracted $\lambda_\Gamma(p_h^l)$ 
and used the fact that $\lambda_\Gammah(p_h) = 0$.

\paragraph{Term $\bfI$.}
Let $(\bfphi,\chi) \in \bfV_t \times Q$ be 
the solution to the continuous dual problem
\begin{alignat}{2}\label{eq:Darcydual}
-\divs \bfphi &= \psi \qquad &\text{on $\Gamma$}
\\ \label{eq:Darcyb}
-\bfphi_t + \nablas \chi &= \bfzero \qquad &\text{on $\Gamma$}
\end{alignat}  
for $\psi \in L^2(\Gamma)$ with $\int_\Gamma \psi =0$.  We then 
have the elliptic regularity bound
\begin{equation}\label{eq:contdualstab}
\textcolor{black}{\|\bfphi\|_{H^1(\Gamma)} + \| \chi \|_{H^2(\Gamma)} }
\lesssim \|\psi\|_\Gamma
\end{equation}
Furthermore, we have the weak form
\begin{equation}
A((\bfv,q),(\bfphi,\chi)) = (q,\psi)  \qquad \forall (\bfv,q) 
\in \bfV \times Q
\end{equation}
and setting $(\bfv,q) = (\bfu - \bfu_h^l,p-p_h^l)$ we obtain 
the error representation formula 
\begin{align}
(p-p_h^l,\psi) &= A((\bfu- \bfu_h^l,p - p_h^l),(\bfphi,\chi)) 
\\
&=A((\bfu- \bfu_h^l,p - p_h^l),(\bfphi,\chi) - (\pi_h\bfphi,\pi_h\chi)) 
\\ \nonumber
&\quad + A((\bfu- \bfu_h^l,p - p_h^l),(\pi_h\bfphi,\pi_h\chi))
\\
&=A((\bfu- \bfu_h^l,p - p_h^l),(\bfphi,\chi) - (\pi_h\bfphi,\pi_h\chi)) 
\\ \nonumber
&\quad + L((\pi_h\bfphi,\pi_h\chi)) - L_h((\pi_h\bfphi,\pi_h\chi))
\\ \nonumber
&\quad
+ A_h((\bfu_h,p_h),(\pi_h\bfphi,\pi_h\chi))
- A((\bfu_h,p_h)^l,(\pi_h\bfphi,\pi_h\chi)^l)
\\
&=I_1 + I_2 + I_3
\end{align}
\paragraph{Term $\bfI_1$.} Using the continuity of $A$, the energy 
error estimate (\ref{eq:energyest}), and the interpolation error 
estimate (\ref{eq:interpol-energy}) we obtain
\begin{align}
|I| &\lesssim \tn (\bfu- \bfu_h^l,p - p_h^l) \tn \, \tn (\bfphi,\chi) - (\pi_h\bfphi,\pi_h\chi)^l \tn
\\
&\lesssim 
\Big(h^{k_u+1}\| \bfu \|_{H^{k_u+1}(\Gamma)} 
+ h^{k_p} \| p \|_{H^{k_p+1}(\Gamma)} \Big)
\Big( h \| \bfphi \|_{H^{1}(\Gamma)} 
+ h \| \chi \|_{H^{2}(\Gamma)} \Big)
\\
&\lesssim 
\Big(h^{k_u+2}\| \bfu \|_{H^{k_u+1}(\Gamma)} 
+ h^{k_p+1} \| p \|_{H^{k_p+1}(\Gamma)} \Big)
\|\psi\|_\Gamma
\end{align} 

\paragraph{Term $\bfI_2$.} Using the quadrature estimate 
(\ref{eq:quadestLh}) we directly obtain
\begin{align}
|II| &\lesssim  
h^{k_g+1} \Big( \| f \|_\Gamma + \|\bfg \|_\Gamma\Big)
\tn (\pi_h\bfphi,\pi_h\chi) \tn_h
\lesssim 
h^{k_g+1} \Big( \| f \|_\Gamma + \|\bfg \|_\Gamma\Big) \|\psi\|_\Gamma
\end{align}
where we used stability (\ref{eq:interpolstab}) of the interpolant and stability (\ref{eq:contdualstab}) of the solution to the dual problem.

\paragraph{Term $\bfI_3$.} Using the quadrature estimate 
(\ref{eq:quadestAh}) we obtain
\begin{align}
|III| &\lesssim 
h^{k_g+1} \Big(  \tn (\bfu_h,p_h) \tn_h + h^{-1} \| \bfn \cdot \bfu_h \|_\Gammah \Big)
\Big( \tn (\pi_h\bfphi,\pi_h\chi) \tn_h + h^{-1} \| \bfn \cdot \pi_h \bfphi \|_\Gammah \Big)
\\
&\lesssim 
h^{k_g+1}(\|f\|_\Gamma + \| \bfg \|_\Gamma ) 
( \tn (\bfphi,\chi ) \tn + \|\phi \|_{H^1(\Gamma)} ) 
 \\
&\lesssim 
h^{k_g+1}(\|f\|_\Gamma + \| \bfg \|_\Gamma ) 
\|\psi \|_\Gamma
\end{align}
where we used, the stability (\ref{eq:interpolstab}) of the interpolant,  
the estimate (\ref{eq:interpol-normal}) for the normal component of the 
interpolant, and the stability (\ref{eq:contdualstab}) of the dual problem.

\paragraph{Conclusion Term $\bfI$.} Finally, setting 
$\psi = p-(p_h^l - \lambda_\Gamma(p_h^l))/ \|p-(p_h^l - \lambda_\Gamma(p_h^l))\|_\Gamma$, and collecting the estimates of terms $I_1$, $I_2$, and $I_3$, we obtain
\begin{equation}\label{eq:L2pTerm1}
I \lesssim   
h^{k_u+2}\| \bfu \|_{H^{k_u+1}(\Gamma)} 
+ h^{k_p+1} \| p \|_{H^{k_p+1}(\Gamma)}
+ h^{k_g+1} (\|f\|_\Gamma + \| \bfg \|_\Gamma ) 
\end{equation}

\paragraph{Term $\bfI\bfI$.} Changing the domain of 
integration from $\Gamma$ to $\Gammah$, we obtain the 
identity
\begin{align}
\lambda_\Gamma(p_h^l) -  \lambda_\Gammah(p_h)
&= 
|\Gamma|^{-1} \int_\Gamma p_h^l - |\Gammah|^{-1} \int_\Gammah p_h
\\
&=\int_{\Gammah} ( |\Gamma|^{-1} |\bfB| - |\Gammah|^{-1} ) p_h
\end{align}
Using the estimates $|\bfB| = 1 + O(h^{k_g+1})$ 
and $|\Gammah| = |\Gamma| + O(h^{k_g+1})$, and some obvious 
manipulations we obtain
\begin{align}
|\lambda_\Gamma(p_h^l) -  \lambda_\Gammah(p_h)|
&\lesssim 
\| |\Gamma|^{-1} |\bfB| - |\Gammah|^{-1} \|_{L^\infty(\Gammah)} 
\| p_h \|_\Gammah
\\ 
&\qquad \lesssim h^{k_g+1} \| \nablash p_h \|_\Gammah 
\lesssim h^{k_g+1} (\|f\|_\Gamma + \|\bfg \|_\Gamma ) 
\end{align}
where at last we used the Poincar\'e estimate (\ref{PoincareGammah}). 
Thus we conclude that 
\begin{equation}\label{eq:L2pTerm2}
II \lesssim h^{k_g+1} (\|f\|_\Gamma + \|\bfg \|_\Gamma ) 
\end{equation}

\paragraph{Conclusion.} Together the estimates (\ref{eq:L2pTerm1}) 
and (\ref{eq:L2pTerm2}) of Terms $I$ and $II$ proves the desired estimate.
\end{proof}

\begin{thm}\label{thm:L2estu} Under the same assumptions as 
in Theorem \ref{thm:energyest} the following estimate holds
\begin{align}\nonumber
\|\bfPs( \bfu - \bfu_h^l) \|_{\Gamma} 
&\lesssim h^{k_u+1}\| \bfu \|_{H^{k_u+1}(\Gamma)} 
+ h^{k_p} \| p \|_{H^{k_p+1}(\Gamma)}
+ h^{k_g+1} (\|f\|_\Gamma + \|\bfg\|_\Gamma )
\end{align}
for all $h\in (0,h_0]$. 
\end{thm}
\begin{proof} Let $(\bfphi_h,\chi_h)$ be the solution to the 
discrete dual problem
\begin{equation}\label{eq:discretedual}
A_h((\bfv,q), (\bfphi_h,\chi_h) ) = (\bfpsi^e,\bfv)_\Gammah
\end{equation}
where $\bfpsi:\Gamma \rightarrow \IR^3$ is a given tangential vector 
field. We note that there is a unique solution to (\ref{eq:discretedual}) , 
and using the technique in the proof of Lemma 
\ref{lem:stab}, we conclude that the following stability estimate holds
\begin{equation}\label{eq:discretedualstability}
\tn (\bfphi_h,\chi_h)\tn_h +h^{-1} \| \bfn \cdot \bfphi_h \|_\Gammah 
\lesssim 
\| \bfpsi \|_\Gamma
\end{equation}

Setting $(\bfv,q) = ( \pi_h \bfu - \bfu_h, \pi_h p - p_h) $ 
in (\ref{eq:discretedual})  we obtain the error representation formula 
\begin{align}
(\pi_h \bfu - \bfu_h,\bfpsi^e)_\Gammah 
&=  
A_h((\pi_h \bfu - \bfu_h, \pi_h p - p_h), (\bfphi_h,\chi_h) ) 
\\
&=
A_h((\pi_h \bfu , \pi_h p ), (\bfphi_h,\chi_h) ) - L_h( (\bfphi_h,\chi_h) )
\\
&=
A_h((\pi_h \bfu , \pi_h p ), (\bfphi_h,\chi_h) ) 
\\ \nonumber
&\qquad 
\underbrace{- A(( \bfu ,  p ), (\bfphi_h,\chi_h)^l ) 
+
L( (\bfphi_h,\chi_h)^l )}_{=0}
- 
L_h( (\bfphi_h,\chi_h) )
\\
&=
{A(( \pi_h \bfu , \pi_h p  )^l - (\bfu,p), (\bfphi_h,\chi_h)^l )}
\\ \nonumber
&\qquad 
+A_h((\pi_h \bfu , \pi_h p ), (\bfphi_h,\chi_h) ) 
- A(( \pi_h \bfu , \pi_h p )^l, (\bfphi_h,\chi_h)^l )
\\
&\qquad
+L( (\bfphi_h,\chi_h)^l )
- 
L_h( (\bfphi_h,\chi_h) )
\\
&=I+II+III
\end{align}
The terms may be estimated as follows.

\paragraph{Term $\bfI$.} We have
\begin{align}
|I|
&\lesssim 
\tn (\bfu - (\pi_h \bfu)^l, p - (\pi_h p)^l)\tn \, \tn (\bfphi_h,\chi_h)\tn_h
\\
&\lesssim 
\Big( h^{k_u+1} \| \bfu \|_{H^{k_u+1}(\Gamma)}   
+h^{k_p} \| p \|_{H^{k_p+1}(\Gamma)} \Big) \|\bfpsi\|_{\Gamma}  
\end{align}
where we used continuity of $A_h$, the interpolation estimate (\ref{eq:interpol}), 
and the stability estimate (\ref{eq:discretedualstability}).

\paragraph{Term $\bfI\bfI$.} We have
\begin{align}
|II| 
&\lesssim 
h^{k_g + 1} 
\Big(\tn (\pi_h \bfu , \pi_h p ) \tn_h  
+h^{-1} \| \bfn \cdot \pi_h \bfu \|_\Gammah \Big)
\Big( \tn (\bfphi_h,\chi_h) \tn_h +h^{-1} \| \bfn \cdot \bfphi_h \|_\Gammah \Big)
\\
&\lesssim 
h^{k_g + 1} (\|\bfu \|_{H^1(\Gamma)} + \| p \|_{H^1(\Gamma}) 
\| \bfpsi \|_\Gamma
\end{align}
where we used the quadrature estimate (\ref{eq:quadestAh}), the stability of the interpolation operator, and the stability estimate (\ref{eq:discretedualstability}).

\paragraph{Term $\bfI\bfI\bfI$.} We have
\begin{align}
|III|
&\lesssim 
h^{k_g + 1} (\|f\|_\Gamma + \| \bfg \|_\Gamma ) \tn (\bfphi_h,\chi_h) \|_h
\\
&\lesssim 
h^{k_g + 1} (\|f\|_\Gamma + \| \bfg \|_\Gamma ) \|\bfpsi\|_\Gamma
\end{align}
where we used the quadrature estimate (\ref{eq:quadestLh}) and the stability estimate (\ref{eq:discretedualstability}).

\paragraph{Conclusion.} Finally, setting $\bfpsi = \bfPs \bfe_{h,u}/  \| \bfPs \bfe_{h,u}\|_\Gamma$ we have 
\begin{equation}
(\bfe_{h,u},\bfpsi)_\Gammah = \|\bfPs \bfe_{h,u}\|_\Gammah
\sim  \|\bfPs (\bfu - \bfu_h^l ) \|_\Gammah
\end{equation}
where we used equivalence of norms (\ref{eq:normequ}). Thus the proof is complete.
\end{proof}

\section{Numerical Example: Flow on a Torus}

Let $\Gamma$ be the  torus, given implicitly by the solution to
\[
 (R-\sqrt{x^2 + y^2})^2 + z^2 -r^2 = 0
\]
were $R=1$ is the major and $r=1/2$ is the minor radius and 
let the right-hand side $\bfg$ correspond to the solution
\[
\bfu_t  = \left(2 x z ,-2 y z,2(x^2 - y^2)(R-\sqrt{x^2 + y^2})/\sqrt{x^2 + y^2}\right),\quad u_n=0,\quad p= z .
\]
Note that $\divs\bfu = 0$, so $f=0$. 
The errors are computed on the discrete geometry by defining
$e_p:=\| p^e-p_h\|_{\Gamma_h}$ and $e_{u}:=\| \bfu^e - \bfu_h\|_{\Gamma_h}$.
For the evaluation of the integral $(\bfg_h,\bfv)_{\Gamma_h}$, we use $\bfg_h = \bfg^e$. We have used $c_N=0$, see Remark \ref{remark-cn}. We emphasize that $c_N> 0$ does not affect the asymptotic convergence rate; however, for large $c_N$ a locking effect can occur. Moderate sizes of $c_N$ have a negligible 
effect on the error.

In Fig. \ref{fig:torus} we show the computed velocity field on a particular mesh, with computed pressure isolevels on
the same mesh in Fig. \ref{fig:press}.

In order to make a comparison between structured and unstructured meshes, we create a sequence of unstructured meshes by randomly moving the nodes on each mesh in a sequence of structured meshes. A typical example is shown in Fig. \ref{fig:unstruct}.
We then make eight comparisons:
\begin{table}
\centering
\begin{tabular}{|c c c c c |c c|}
\hline
Case & $k_u$ & $k_p$ & $k_g$ & Mesh & Order $e_u$ & Order $e_p$ 
\\
\hline
\hline
1     & 1 & 1 & 1 & Structured & 2 & 2
\\
\hline
2     & 1 & 2 & 1 & Structured & 2 & 2
\\
\hline
3     & 1 & 1 & 1 & Unstructured & 1 & 2
\\
\hline
4     & 1 & 2 & 1 & Unstructured & 2 & 2
\\
\hline
5     & 1 & 1 & 2 & Structured & 2 & 2
\\
\hline
6     & 1 & 2 & 2 & Structured & 2 & 3 
\\
\hline
7     & 1 & 1 & 2 & Unstructured & 1 & 2
\\
\hline
8     & 1 & 2 & 2 & Unstructured & 2 & 3 
\\
\hline
\end{tabular}
\caption{The eight  cases considered in the numerical examples together with the observed orders of convergence.}
\end{table}

We note the following: for Case 1 (Fig. \ref{fig:linearstruct}), the geometry error does not affect the solution and we get optimal $O(h^2)$-convergence in pressure and superconvergence $O(h^2)$--convergence of velocities, related to the structuredness of the meshes. In Case 2 (Fig. \ref{fig:linearstruct}), the increase in polynomial degree for the pressure does not help because of the poor geometry approximation and we keep $O(h^2)$--convergence.

For Case 3 (Fig. \ref{fig:linearunstruct}) we lose the superconvergence in velocity, which becomes $O(h)$; for Case 4 (Fig. \ref{fig:linearunstruct}) we regain optimal convergence for the velocity of $O(h^2)$ but, as for Case 2, the geometry approximation
precludes optimal convergence of the pressure which remains $O(h^2)$.

For Case 5 (Fig. \ref{fig:quadstruct})
we have optimal convergence in pressure of $O(h^2)$ and superconvergence of velocity of $O(h^2)$. Improving the pressure
approximation as in Case 6 (Fig. \ref{fig:quadstruct}) now leads to the expected convergence of $O(h^2)$ for velocity and $O(h^3)$ for pressure. 

On unstructured meshes, Case 7 (Fig. \ref{fig:quadunstruct}) we lose the superconvergence of velocities and obtain only $O(h)$ in velocity error with linear pressures, while Case 8  (Fig. \ref{fig:quadunstruct}) again gives the expected error of $O(h^2)$ for velocity and $O(h^3)$ for pressure.

We conclude that:
 \begin{itemize}
\item Piecewise linear approximations for $\bfu_h$ and $p_h$ have superconvergence of velocity on \emph{structured}\/ meshes, both for piecewise linear and piecewise quadratic geometry.
\item Increasing the polynomial degree of the pressure to $P^2$ increases the convergence of a piecewise linear velocity from $O(h)$ to $O(h^2)$ on \emph{unstructured}\/ meshes, both for piecewise linear and piecewise quadratic geometry.
\item Increasing the convergence rate of the pressure, from $O(h^2)$ to $O(h^3)$, when going from $P^1$ to $P^2$ approximations requires
that the same increase is being made in the geometry approximation.
 \end{itemize}
Comments:
\begin{itemize}
\item For unstructured meshes these results are in accordance with the theoretical investigations in Section 5. We note, however, that in the numerical results the normal component of the error in the velocity 
also converges optimally, i.e. of order $k_g+1$, with respect to the order of approximation of the geometry while in Theorem \ref{thm:energyest} 
we achieve order $k_g$. In Theorem \ref{thm:L2estu} we, however, 
show that the tangent component of the error is indeed optimal with respect to the order of the approximation of the geometry. Thus our 
theoretical results are in line with the numerical results but slightly weaker with respect to the order of approximation of the geometry 
for the normal component of the error. The error in the pressure and 
the tangent component are optimal with respect to $k_p$, $k_u$, 
and $k_g$.

\item For structured meshes the superconvergence most certainly is
related to superconvergence of $L^2$ projections of the gradient on 
the continuous space, see \cite{BaXu03}, which holds on structured 
meshes.
\end{itemize}

\bibliographystyle{plain}
  \bibliography{ref}

\begin{thebibliography}{10}

\bibitem{BaXu03}
R.~E. Bank and J.~Xu.
\newblock Asymptotically exact a posteriori error estimators. {I}. {G}rids with
  superconvergence.
\newblock {\em SIAM J. Numer. Anal.}, 41(6):2294--2312 (electronic), 2003.

\bibitem{BrSc08}
S.~C. Brenner and L.~R. Scott.
\newblock {\em The Mathematical Theory of Finite Element Methods}, volume~15 of
  {\em Texts in Applied Mathematics}.
\newblock Springer, New York, third edition, 2008.

\bibitem{BuHaLa14}
E.~Burman, P.~Hansbo, and M.~G. Larson.
\newblock A stabilized cut finite element method for partial differential
  equations on surfaces: the {L}aplace-{B}eltrami operator.
\newblock {\em Comput. Methods Appl. Mech. Engrg.}, 285:188--207, 2015.

\bibitem{De09}
A.~Demlow.
\newblock Higher-order finite element methods and pointwise error estimates for
  elliptic problems on surfaces.
\newblock {\em SIAM J. Numer. Anal.}, 47(2):805--827, 2009.

\bibitem{Dz88}
G.~Dziuk.
\newblock Finite elements for the {B}eltrami operator on arbitrary surfaces.
\newblock In {\em Partial differential equations and calculus of variations},
  volume 1357 of {\em Lecture Notes in Math.}, pages 142--155. Springer,
  Berlin, 1988.

\bibitem{DzEl13}
G.~Dziuk and C.~M. Elliott.
\newblock Finite element methods for surface {PDE}s.
\newblock {\em Acta Numer.}, 22:289--396, 2013.

\bibitem{FeFoFu14}
A.~Feronni, L.~Formaggia, and A.~Fumagalli.
\newblock Numerical analysis of {D}arcy problem on surfaces.
\newblock Technical Report~30, MOX, Dipartimento di Matematica ``F. Brioschi''
  Politecnico di Milano, Via Bonardi 9 - 20133 Milano (Italy), 2014.

\bibitem{GiTr01}
D.~Gilbarg and N.~S. Trudinger.
\newblock {\em Elliptic Partial Differential Equations of Second Order}.
\newblock Classics in Mathematics. Springer-Verlag, Berlin, 2001.
\newblock Reprint of the 1998 edition.

\bibitem{HaLa14}
P.~Hansbo and M.~G. Larson.
\newblock Finite element modeling of a linear membrane shell problem using
  tangential differential calculus.
\newblock {\em Comput. Methods Appl. Mech. Engrg.}, 270:1--14, 2014.

\bibitem{HaLaLa14}
P.~Hansbo, M.~G. Larson, and F.~Larsson.
\newblock Tangential differential calculus and the finite element modeling of a
  large deformation elastic membrane problem.
\newblock {\em Comput. Mech.}, 56(1):87--95, 2015.

\bibitem{LaLa14}
K.~Larsson and M.~G. Larson.
\newblock A continuous/discontinuous {G}alerkin method and a priori error
  estimates for the biharmonic problem on surfaces.
\newblock Technical report, Ume{\aa} University, Department of Mathematics,
  SE-90187, Ume\aa, Sweden, 2015.
\newblock arXiv:1305.2740v2.

\bibitem{MaHu02}
A.~Masud and T.~J.~R. Hughes.
\newblock A stabilized mixed finite element method for {D}arcy flow.
\newblock {\em Comput. Methods Appl. Mech. Engrg.}, 191(39-40):4341--4370,
  2002.

\end{thebibliography}
  
\newpage
\begin{figure}
\begin{center} 
\includegraphics[width=12cm]{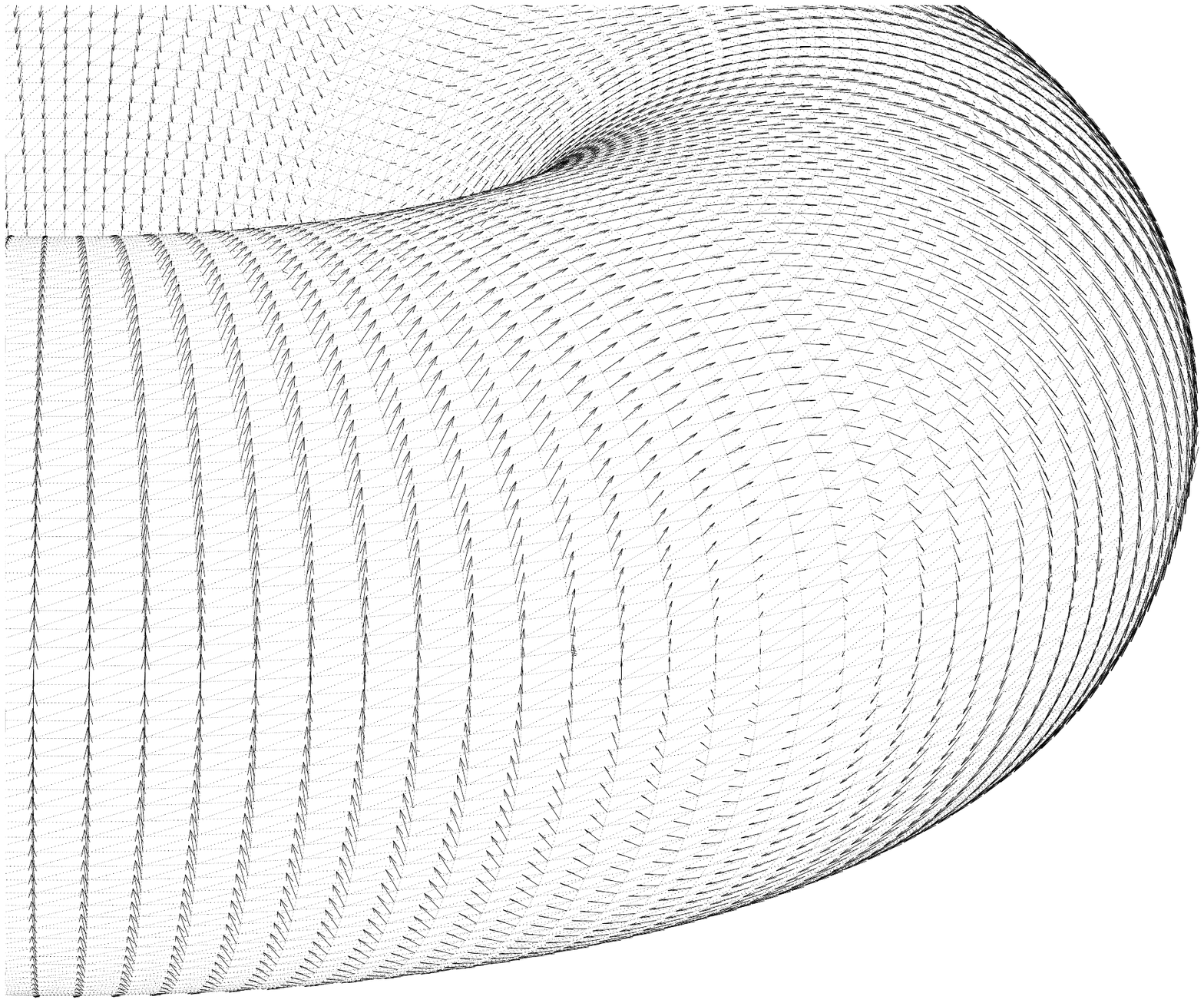}
\caption{A detail of the computed velocity on a structured mesh. \label{fig:torus}}
\end{center}
\end{figure}
\begin{figure}
\begin{center} 
\includegraphics[width=15cm]{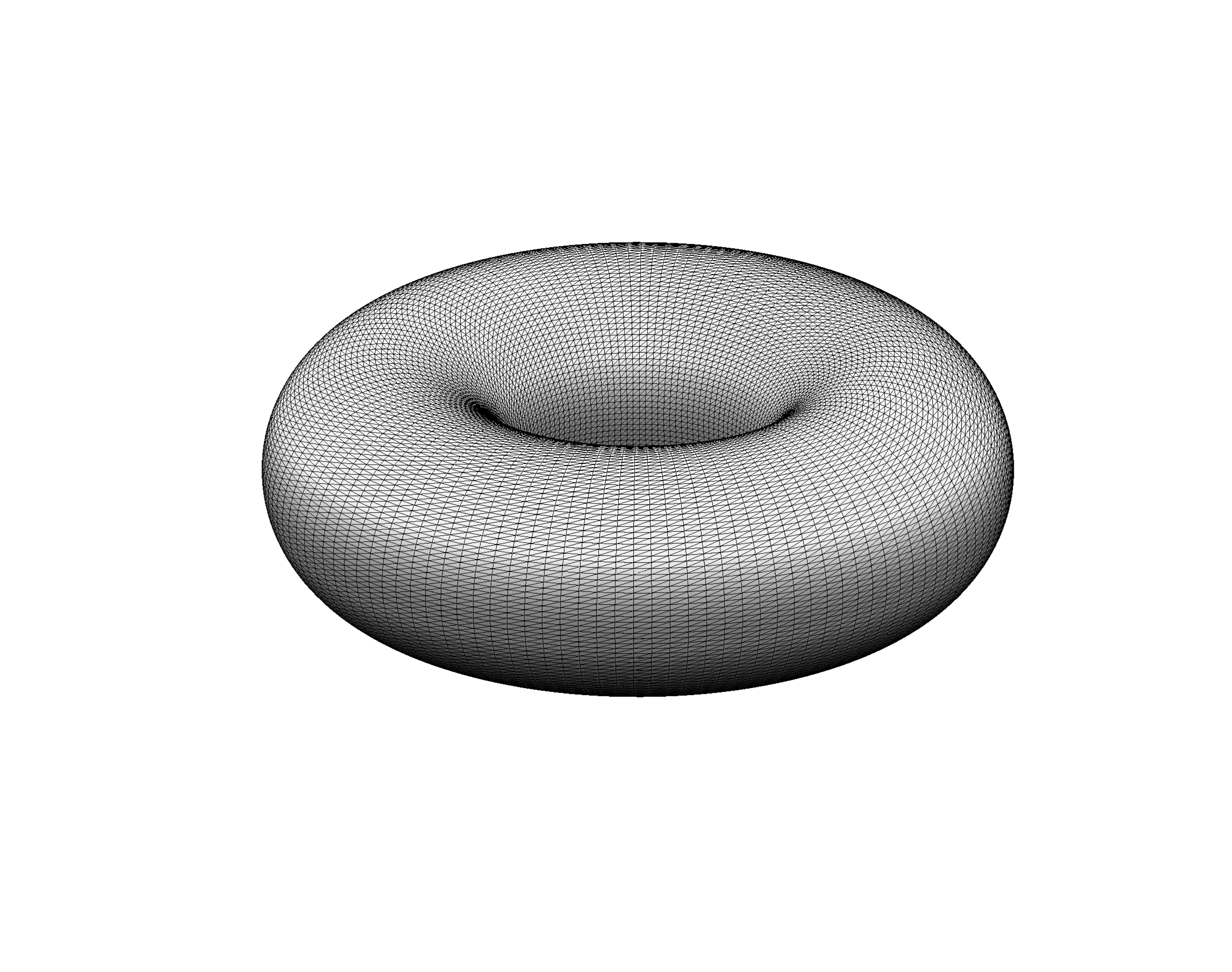}
\caption{Computed pressure on a structured mesh. \label{fig:press}}
\end{center}
\end{figure}
\begin{figure}
\begin{center} 
\includegraphics[width=15cm]{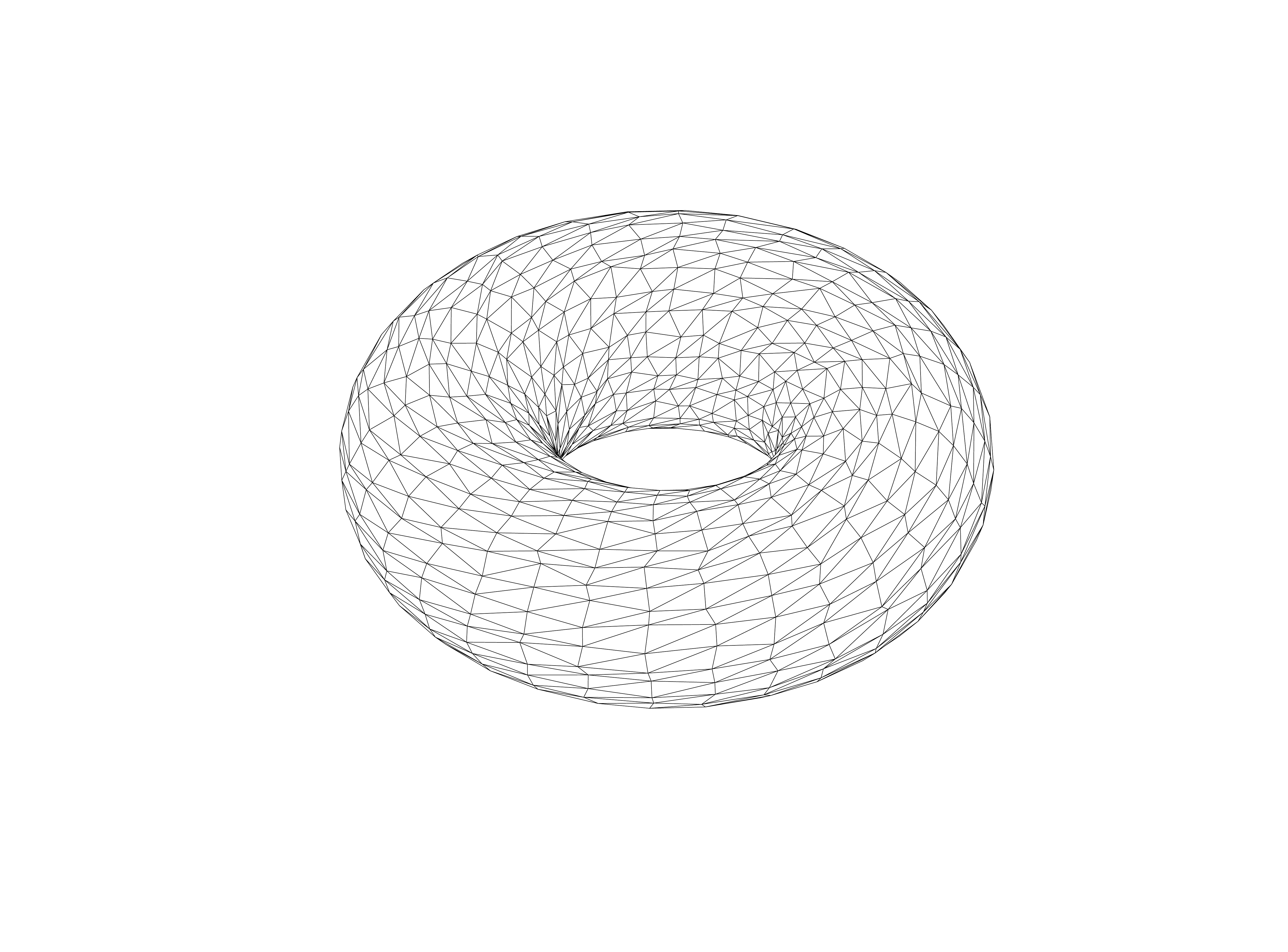}
\caption{Unstructured mesh acquired by jiggling the nodes of a structured mesh. \label{fig:unstruct}}
\end{center}
\end{figure}
\begin{figure}
\begin{center} 
\includegraphics[width=15cm]{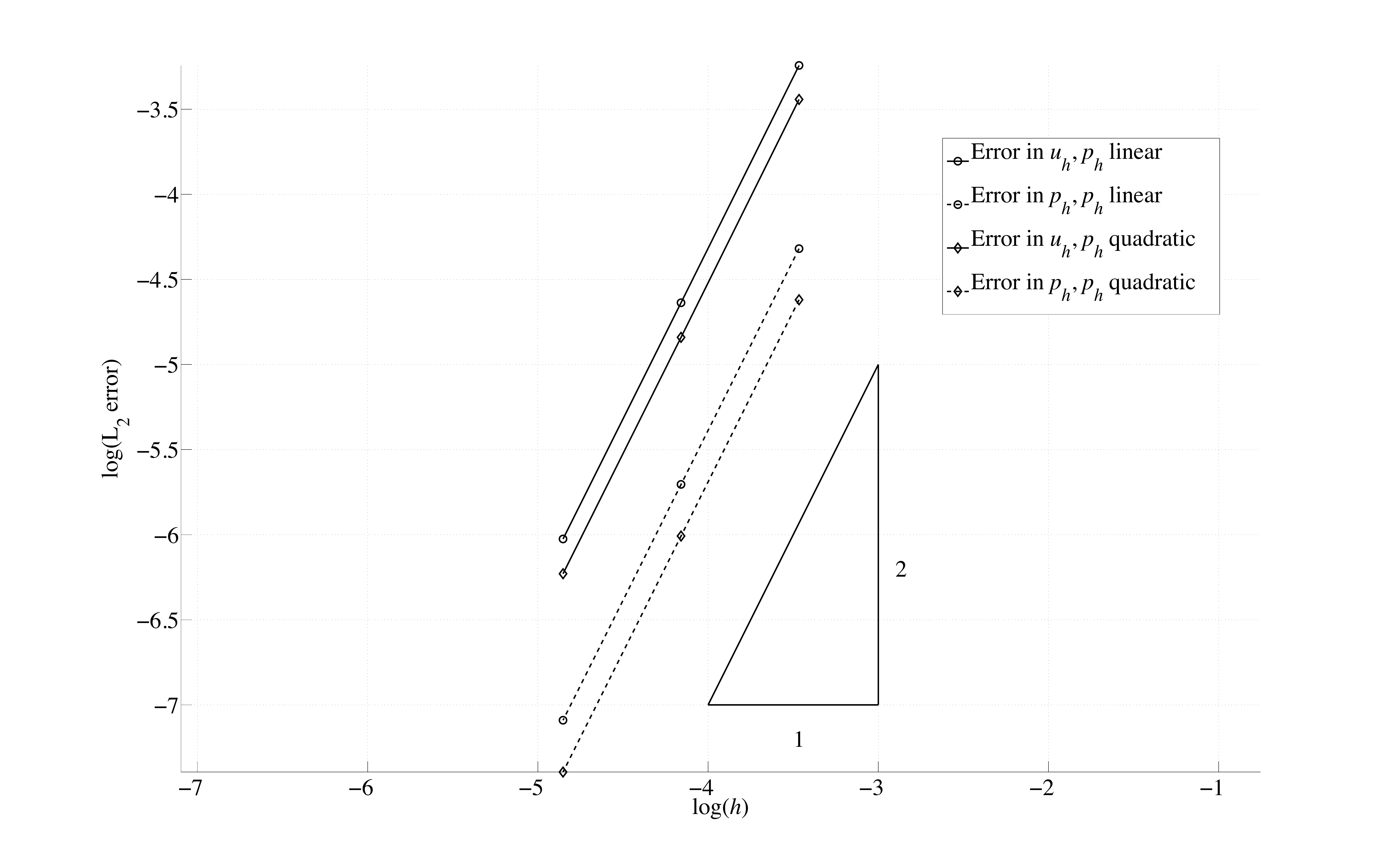}
\caption{Computed errors on a piecewise linear, structured mesh. \label{fig:linearstruct}}
\end{center}
\end{figure}
\begin{figure}
\begin{center} 
\includegraphics[width=15cm]{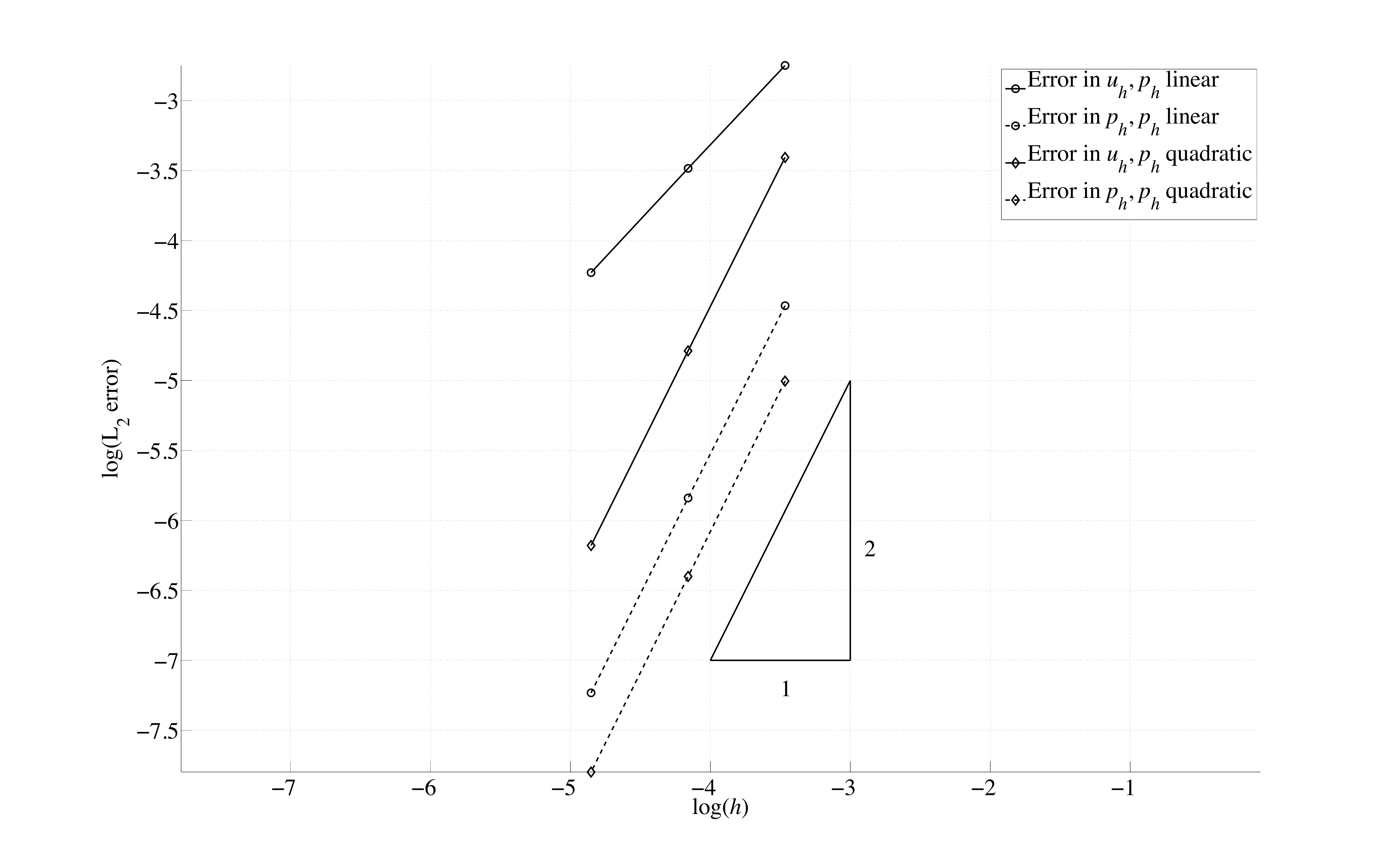}
\caption{Computed errors on a piecewise linear, unstructured mesh. \label{fig:linearunstruct}}
\end{center}
\end{figure}
\begin{figure}
\begin{center} 
\includegraphics[width=15cm]{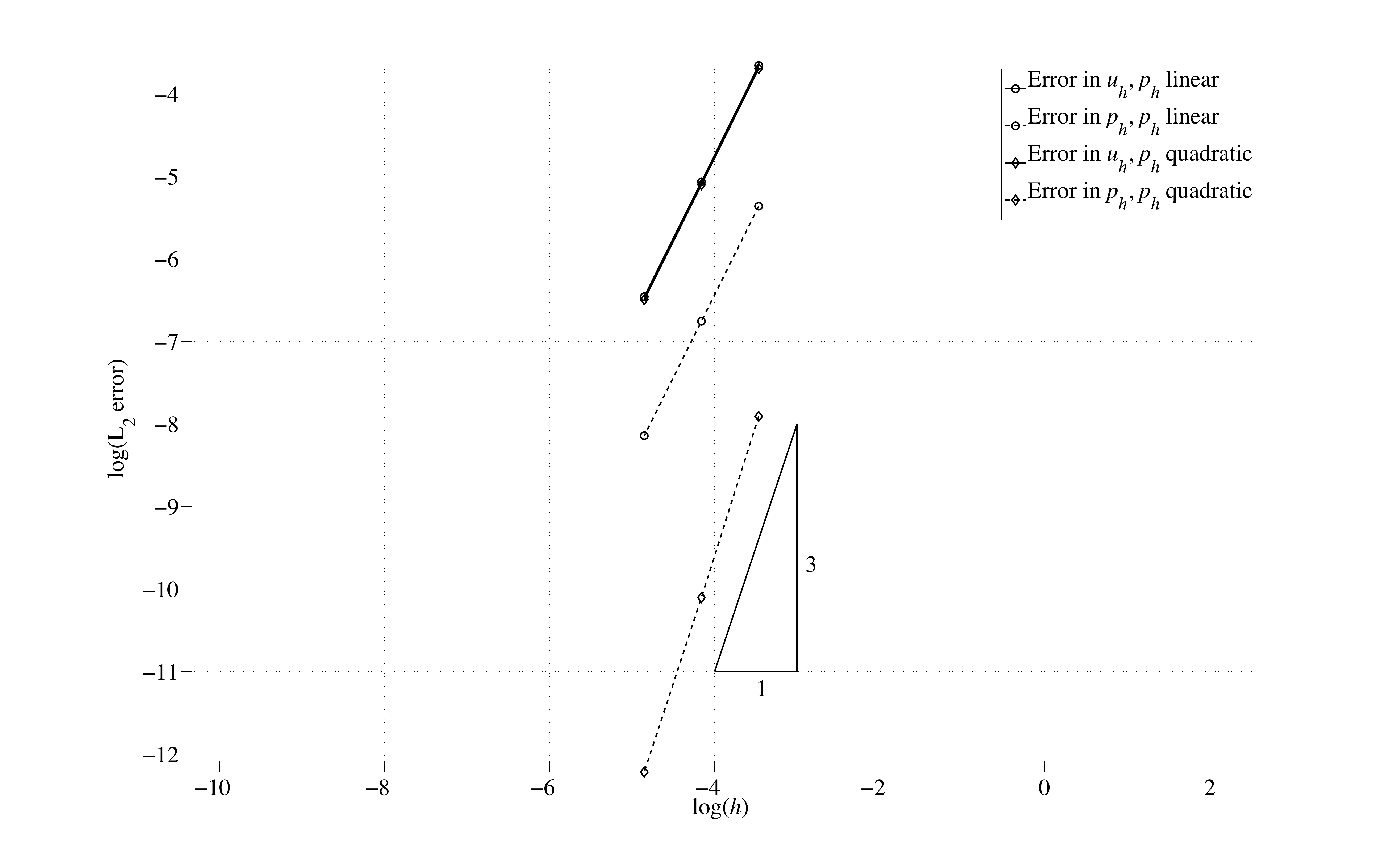}
\caption{Computed errors on a piecewise quadratic, structured mesh. \label{fig:quadstruct}}
\end{center}
\end{figure}
\begin{figure}
\begin{center} 
\includegraphics[width=15cm]{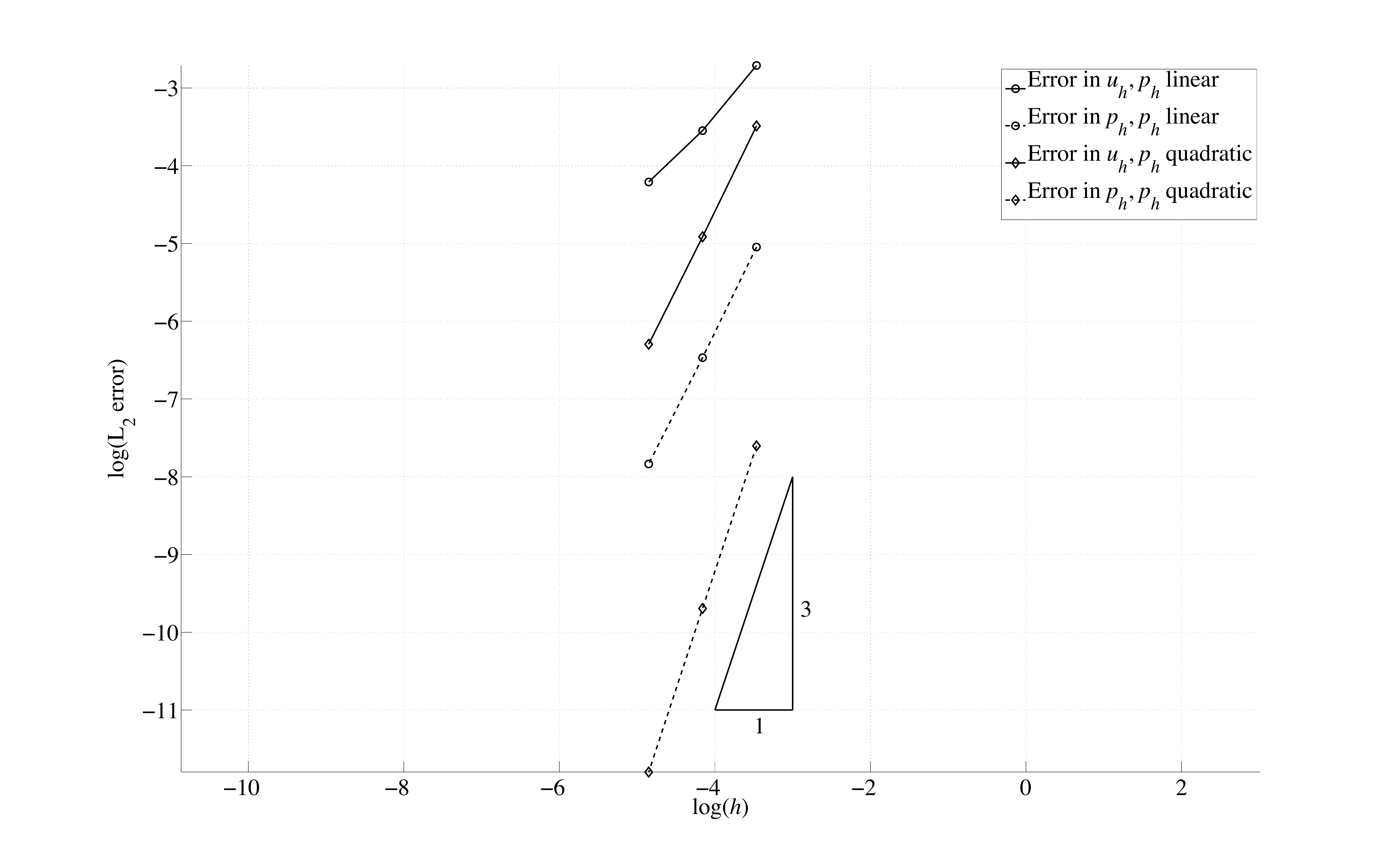}
\caption{Computed errors on a piecewise quadratic, unstructured mesh. \label{fig:quadunstruct}}
\end{center}
\end{figure}

\end{document}